\documentclass[10pt]{amsart}

\usepackage[T1]{fontenc}
\usepackage[utf8]{inputenc}
\usepackage[english]{babel}
\usepackage[babel]{microtype}

\usepackage{amsmath, amssymb, amsthm, amsfonts}

\usepackage{mathtools, mathrsfs, dsfont, stmaryrd, old-arrows}

\usepackage{graphicx, xcolor}

\usepackage{tikz-cd}
\usetikzlibrary{babel}
\usepackage{quiver}

\usepackage[pdfencoding = auto, psdextra, colorlinks = true, linkcolor = ., citecolor = ., urlcolor = blue]{hyperref}
\usepackage[bottom]{footmisc}
\usepackage[backend=biber,style=alphabetic,maxbibnames=99]{biblatex}

\usepackage{lipsum}
\usepackage[l2tabu, orthodox]{nag}
\usepackage{relsize}
\usepackage{csquotes}
\usepackage[shortlabels]{enumitem}
\usepackage{chngcntr}
\usepackage{multicol}

\numberwithin{equation}{section}

\theoremstyle{plain}

\newtheorem{theorem}{Theorem}[section]
\newtheorem{proposition}[theorem]{Proposition}
\newtheorem{lemma}[theorem]{Lemma}

\newtheorem{corollary}[theorem]{Corollary}

\newtheorem*{theorem*}{Theorem}
\newtheorem*{proposition*}{Proposition}
\newtheorem*{lemma*}{Lemma}
\newtheorem*{claim*}{Claim}
\newtheorem*{corollary*}{Corollary}
\newtheorem*{observation*}{Observation}

\theoremstyle{definition}

\newtheorem{definition}[theorem]{Definition}
\newtheorem{remark}[theorem]{Remark}
\newtheorem{notation}[theorem]{Notation}
\newtheorem{example}[theorem]{Example}

\newtheorem{problem}{Problem}

\newtheorem*{definition*}{Definition}
\newtheorem*{remark*}{Remark}
\newtheorem*{example*}{Example}
\newtheorem*{question*}{Question}
\newtheorem*{problem*}{Problem}
\newtheorem*{conjecture*}{Conjecture}

\let\Im\relax\DeclareMathOperator{\Im}{Im}

\DeclareMathOperator{\Hom}{Hom}
\DeclareMathOperator{\End}{End}

\newcommand*{\Id}{\mrm{Id}}

\newcommand*{\ct}[1]{\mrm{#1}}

\newcommand*{\hyp}[2]{#1\text{-}#2}

\newcommand*{\cthyp}[2]{\hyp{\ct{#1}}{#2}}

\newcommand*{\bb}{\mathbb}

\newcommand*{\mrm}{\mathrm}

\newcommand*{\tox}{\xrightarrow}

\newcommand*{\Gr}{\ct{Grp}}

\newcommand*{\ModR}[1][R]{\cthyp{Mod}{S}}

\newcommand*{\RGr}[1][R]{#1 \hspace*{.05em} \Gr}

\DeclareMathOperator{\Poly}{Poly}

\usepackage[mathcal]{euscript}

\let\oldtocsection=\tocsection
\let\oldtocsubsection=\tocsubsection
\let\oldtocsubsubsection=\tocsubsubsection

\renewcommand{\tocsection}[2]{\hspace{0em}\oldtocsection{#1}{#2}}
\renewcommand{\tocsubsection}[2]{\hspace{1em}\oldtocsubsection{#1}{#2}}
\renewcommand{\tocsubsubsection}[2]{\hspace{2em}\oldtocsubsubsection{#1}{#2}}

\begin{document}

\title{Ring-induced localizations of nilpotent groups}

\author[S.O.Ivanov]{Sergei O. Ivanov} 
\email{ivanov.s.o.1986@gmail.com, ivanov.s.o.1986@bimsa.cn}
\address{Beijing Key Laboratory of Topological Statistics and Applications for Complex Systems, Beijing Institute of Mathematical Sciences and Applications (BIMSA)}

\author[G. Kadantsev]{Georgii Kadantsev}
\email{kadantsev.georg@yandex.ru}
\address{Qiuzhen College, Tsinghua University}
\author[A. Krasilnikov]{Aleksandr Krasilnikov}
\thanks{The third author was supported by the Basic Research Program of the National Research University Higher School of Economics.}
\email{kras1lnikoff.av@gmail.com, avkrasilnikov@hse.ru}
\address{National Research University Higher School of Economics}

\dedicatory{To Ilya Chistyakov, our teacher, on the occasion of his 65th birthday}

\begin{abstract}
For a commutative ring $R$, we study the $R$-localization functor on the category of groups, defined as localization with respect to the homomorphism $\mathbb{Z}\to R$. Our main result is that, when $R$ is a binomial ring, the $R$-localization of a nilpotent group is again nilpotent. Taking $R=\mathbb{Z}_p$, the ring of $p$-adic integers, yields a new example of a localization functor that preserves nilpotency. To prove this, we characterize $R$-local groups in terms of $R$-groups in the sense of Myasnikov--Remeslennikov. We call an $R$-group a \emph{Hall--Petresco $R$-group} if it satisfies a version of the Hall--Petresco identity, and show that these form a quasivariety closed under quotients by the center. The crucial input to our main result is that every $R$-local group carries a unique Hall--Petresco $R$-group structure.
\end{abstract}

\maketitle

\tableofcontents

\section*{Introduction}

A localization on the category of groups is a functor $L:\Gr\to\Gr$ equipped with
a natural transformation $\eta:\Id\to L$ such that the two natural transformations
$\eta L, L\eta : L\to L^2$ coincide and are isomorphisms. This paper is part of a
broader effort to understand such localizations and the properties they preserve.
Several works have addressed this subject; we refer the reader
to~\cite{casacuberta2000structures, libman2000cardinality,
bastardas2004localitzacions, chatterji2008guido, ivanov2023cokernel} for a
detailed overview. We focus on a long-standing question posed by Emmanuel Farjoun
in the 1990s: whether every localization of a nilpotent group is again nilpotent.

This problem was recently advanced in~\cite{right-exact}. We call a functor
\emph{weakly right exact} if it preserves reflexive coequalizers (the authors
of~\cite{right-exact} call such functors simply right exact, but we reserve the
term \emph{right exact} for functors that preserve all finite colimits). They
showed that every weakly right exact localization maps nilpotent groups to
nilpotent groups, and that many classical localizations on the category of groups
are weakly right exact; they also exhibited localizations that are not.

For a commutative ring $R$, we say that a group $G$ is \emph{$R$-local} if, for
every $g\in G$, there is a unique homomorphism from the additive group of $R$ to
$G$ sending $1$ to $g$. Every group $G$ admits a universal map $G\to L_R(G)$ to an
$R$-local group, and this defines a localization
\begin{equation}
L_R:\Gr \to \Gr,
\end{equation}
which we call the \emph{$R$-localization}. The simplest known localization on the
category of groups that is not weakly right exact appears to be the
$\mathbb{Z}_p$-localization, where $\mathbb{Z}_p$ is the ring of $p$-adic
integers. This raises a natural question: is the $\mathbb{Z}_p$-localization of a
nilpotent group nilpotent? Our main result answers this affirmatively, in greater
generality.

\begin{theorem*}[{Theorem~\ref{thm:nilpotency-preservation}}]
For a binomial ring $R$, the $R$-localization of a nilpotent group is nilpotent.
\end{theorem*}

In fact, we prove a more general statement: the same conclusion holds whenever
$R$ admits a normed pre-$\lambda$-ring structure.

If $P$ is a set of primes and $R=\mathbb{Z}[P^{-1}]\subseteq \mathbb{Q}$, then
the $R$-localization coincides with Baumslag's $P$-localization, which is known
to preserve reflexive coequalizers. Thus, our main result yields an alternative
proof that Baumslag's $P$-localization preserves the class of nilpotent groups.
It also gives an affirmative answer to a question posed
in~\cite{warfield2006nilpotent}, namely whether the $P$-localization in the
category of nilpotent groups coincides with the $P$-localization in the category
of all groups.

The notion of an $R$-local group is closely related to that of an $R$-group,
introduced by Lyndon~\cite{lyndon_groups_1960} and later refined by
Myasnikov--Remeslennikov~\cite{myasnikov_groups_1994}. An $R$-group is a group
$G$ equipped with an exponentiation map
\begin{equation}
G\times R \to G, \qquad (g,r)\mapsto g^r,
\end{equation}
satisfying a list of axioms (Definition~\ref{def:R-group}). We show that a group
$G$ is $R$-local if and only if it admits an $R$-group structure and
$\Hom(R^+/\langle 1_R\rangle,G)=0$ (Theorem~\ref{thm:r-local-r-groups}).

A disadvantage of general $R$-groups is that they are not closed under
central quotients: an $R$-group structure on $G$ does not necessarily induce
one on the quotient $G/Z(G)$ (Theorem~\ref{th:example:non-ideal}). For a
binomial ring $R$, following ideas of Hall~\cite{nilpotent-groups}, we
circumvent this difficulty by introducing the notion of a
\emph{Hall--Petresco $R$-group} (Definition~\ref{def:hall-petreco}). Recall
that the Hall--Petresco identity expresses products of powers of elements in
a nilpotent group in terms of words whose exponents are binomial
coefficients (Definition~\ref{def:hall-petresco-words}). We call an
$R$-group $G$ a \emph{Hall--Petresco $R$-group} if an appropriate version of
this identity holds for the $R$-powers of elements in every nilpotent
$R$-subgroup of $G$. In contrast to arbitrary $R$-groups, Hall--Petresco
$R$-groups behave well under central quotients: we prove that if $G$ is a
Hall--Petresco $R$-group, then $G/Z(G)$ inherits a Hall--Petresco
$R$-group structure. The crucial step in the proof of our main result is the
fact that every $R$-local group is a Hall--Petresco $R$-group
(Proposition~\ref{prop:local are Hall--Petresco}).

The $R$-groups form a quasi-variety over a suitable
signature~\cite{amaglobeli2023varieties}. When $R$ is a binomial ring, the
Hall--Petresco $R$-groups form a sub-quasi-variety of it that contains every
$R$-local group and is closed under quotients by the center. For an arbitrary
ring $R$, we know of no such sub-quasi-variety, and it is precisely this
obstruction that confines our main result to the binomial ring setting. This
raises the following question.

\begin{problem}
For a commutative ring $R$, is there a sub-quasi-variety of the quasi-variety of
$R$-groups that contains every $R$-local group and is closed under quotients by
the center?
\end{problem}

The categories of $R$-groups and Hall--Petresco $R$-groups are governed by
adjunctions. The forgetful functor $\RGr\to\Gr$ admits a left adjoint
$C_R:\Gr\to\RGr$, the \emph{$R$-completion functor}; likewise, the forgetful
functor $\RGr^{\mathrm{HP}}\to\Gr$ admits a left adjoint
\begin{equation}
C_R^{\mathrm{HP}}:\Gr\to\RGr^{\mathrm{HP}}
\end{equation}
that we call the \emph{Hall--Petresco $R$-completion functor}. As an intermediate
step toward the main theorem, we prove the following.

\begin{proposition*}[Proposition~\ref{prop:F^{(n-1)}(G) is nilpotent}]
For a binomial ring $R$, the Hall--Petresco $R$-completion of a nilpotent group
is nilpotent.
\end{proposition*}

This proposition also bears on the work of Jaikin-Zapirain. The
paper~\cite{jaikin2024free} contained an error at this point: it assumed that the
$R$-completion of a nilpotent group is again nilpotent (see the proof of
\cite[Th.~5.8]{jaikin2024free}). Jaikin-Zapirain corrected this
in~\cite{jaikin2026correction}; we expect that the argument can alternatively be
repaired by replacing the $R$-completion functor with the Hall--Petresco
$R$-completion functor and applying the proposition above, together with
Corollary~\ref{cor:HP:free} and Lemma~\ref{lemma:linfty-hall-completion}.

To prove the Hall--Petresco identities for $R$-local groups, we develop a general
method for establishing polynomial identities in nilpotent groups. For a group
$A$ and a nilpotent group $G$, let $\Poly(A,G)$ denote the group of polynomial
functions from $A$ to $G$. We prove the following theorem.

\begin{theorem*}[Theorem~\ref{theorem:spoly-inclusion}]
Let $f: A \to B$ be a group homomorphism, and let $G$ be a nilpotent group such
that the map
    \begin{equation}
    f^*:\Hom(B,G) \to \Hom(A,G)
    \end{equation}
     is injective. Then the homomorphism
    \begin{equation}
    f^*:\Poly(B, G) \to \Poly(A, G)
    \end{equation}
    is a monomorphism.
\end{theorem*}

The question of whether $R$-localization maps nilpotent groups to nilpotent
groups can be reduced from arbitrary commutative rings to the case of an
$E$-ring. Recall that a commutative ring $R$ is called an $E$-ring if the
evaluation map at $1$ gives an isomorphism
\begin{equation}
\mathrm{End}(R^+) \cong R;
\end{equation}
in other words, $R$ is an $E$-ring if the additive group $R^+$ is $R$-local. For
example, $\mathbb{Z}[P^{-1}]$ and $\mathbb{Z}_p$ are $E$-rings. For every
localization $L$ on the category of groups, the group $L(\mathbb{Z})$ carries a
natural $E$-ring structure, and there is a unique morphism of localizations
$L_{L(\mathbb{Z})}\to L$. We show that for every commutative ring $R$, the
$E$-ring
\begin{equation}
R^e=L_R(\mathbb{Z})
\end{equation}
is a quotient of $R$, and that the natural map $L_{R^e}(G)\to L_R(G)$ is an
epimorphism (Proposition~\ref{prop:epi-L_E->L_R}). Consequently, if the
$R^e$-localization preserves nilpotency, so does the $R$-localization.

If $R$ is binomial, then $R^e$ is also binomial
(Corollary~\ref{cor:R^e-binomial}). However, not every $E$-ring is binomial. We
show that the subring $\mathcal{E}$ of $\mathbb{Z}_7$ defined by
\begin{equation}
\mathcal{E} = \{x \in \mathbb{Z}_7 \mid \exists\, n\geq 0 : 7^n x \in \mathbb{Z}[\sqrt{2}]\}
\end{equation}
is an $E$-ring admitting no normed pre-$\lambda$-ring structure
(Proposition~\ref{prop:seven-saturation}). We also show that the localization
$L_\mathcal{E}$ is not weakly right exact (Corollary~\ref{cor:E-localisation}).
This leaves open the following problem.

\begin{problem}
\label{prob:seven-saturation-preservation}
Does $L_{\mathcal{E}}$ map nilpotent groups to nilpotent groups?
\end{problem}

As an application of our results, we compute the $R$-localizations of free groups
and finitely generated torsion-free nilpotent groups. We show that, for a
torsion-free $E$-ring $R$ and a set $X$, there is an isomorphism
\begin{equation}
L_R(F(X))\cong F_{R}(X),
\end{equation}
where $F_{R}(X)$ is the free $R$-group on $X$ studied by Myasnikov and
Remeslennikov~\cite{myasnikov_groups_1994, myasnikov1996exponential} and by
Jaikin-Zapirain~\cite{jaikin2024free}
(Proposition~\ref{prop:free-e-ring-groups-local}). For a finitely generated
torsion-free nilpotent group $G$ and a binomial ring $R$, we show that $L_R(G)$
coincides with Hall's $R^e$-completion, defined in terms of Mal'cev bases:
\begin{equation}
L_R(G)\cong H_{R^e}(G)
\end{equation}
(see Proposition~\ref{prop:e-ring-localization-hall-completion}). 

In the final section, we develop a method for twisting $R$-group structures:
starting from an $R$-group, it produces new $R$-group structures on the same
underlying group. We apply this method to construct a nilpotent
$\bb Z_p$-group $\mathcal G'$ of class $3$ for which the quotient
$\mathcal G'/Z(\mathcal G')$ does not inherit a $\bb Z_p$-group structure (Theorem~\ref{th:example:non-ideal}).

\section{\texorpdfstring{$R$}{R}-groups}\label{section:r-groups}

In this section we introduce the notions of a weak $R$-group and an $R$-group,
the category $\RGr$ of $R$-groups, and the $R$-completion functor
$C_R:\Gr\to\RGr$.

\begin{notation}
For elements $x,y$ of a group, we define conjugation and the commutator by
\begin{equation}
	x^y = y^{-1} x y, \qquad  [x,y] = x^{-1} y^{-1} xy.
\end{equation}
We will freely use the following formulas:
\begin{equation}
	[xy,z] = [x,z]^y \cdot [y,z], \qquad [x,yz] = [x,z]\cdot [x,y]^z.
\end{equation}
\end{notation}

\begin{definition}[Weak $R$-group]
\label{def:weak-R-group}
Let $R$ be a commutative ring. A \emph{weak $R$-group} (or an $R$-group in the
sense of Lyndon) is a group $G$ equipped with a function
\begin{equation}
	G \times R \to G, \quad (x, r) \mapsto x^r
\end{equation}
satisfying the following axioms:
\begin{itemize}
	\item[(R1)] $x^{r+s} = x^{r} x^{s}$;
	\item[(R2)] $x^1 = x$;
	\item[(R3)] $(x^r)^s = x^{rs}$;
	\item[(R4)] $(x^y)^r = (x^r)^y$,
\end{itemize}
for all $x,y\in G$ and $r,s\in R$. These axioms imply the identities
\begin{equation}
	x^0=1, \qquad (x^r)^{-1}=x^{-r}, \qquad 1^r=1.
\end{equation}
Moreover, it is known that
\begin{equation}
\label{eq:commutativity_of_powers}
	xy = yx \ \ \Rightarrow \ \  x^r y^s = y^s x^r
\end{equation}
for all $r,s\in R$~\cite[Property~2]{myasnikov_groups_1994}. A subgroup of a weak
$R$-group is called a \emph{weak $R$-subgroup} if it is closed under the
$R$-power operations $x\mapsto x^r$.
\end{definition}

\begin{lemma}
\label{lemma:centralizer}
Let $G$ be a weak $R$-group and let $X\subseteq G$ be a subset. Then the
centralizer $\mathsf{Cent}_G(X)$ is a weak $R$-subgroup. In particular, the
center $Z(G)$ is a weak $R$-subgroup.
\end{lemma}
\begin{proof}
This follows from~\eqref{eq:commutativity_of_powers}: if $g$ commutes with every
element of $X$, then so does $g^r$ for every $r\in R$.
\end{proof}

\begin{proposition}
\label{prop:commutative}
Let $G$ be a weak $R$-group and let $X\subseteq G$ be a set of pairwise commuting
elements. Then the weak $R$-subgroup $\langle X \rangle_R$ generated by $X$ is
abelian.
\end{proposition}
\begin{proof}
Since the elements of $X$ commute pairwise, $X\subseteq \mathsf{Cent}_G(X)$, and
hence $\mathsf{Cent}_G(\mathsf{Cent}_G(X)) \subseteq \mathsf{Cent}_G(X)$. In
particular $\mathsf{Cent}_G(\mathsf{Cent}_G(X))$ contains $X$ and is abelian, as
any two of its elements both lie in $\mathsf{Cent}_G(X)$ and so commute. By
Lemma~\ref{lemma:centralizer}, $\mathsf{Cent}_G(\mathsf{Cent}_G(X))$ is a weak
$R$-subgroup; therefore it contains $\langle X \rangle_R$, which is consequently
abelian.
\end{proof}

\begin{definition}[$R$-homomorphism]
An \emph{$R$-homomorphism} of weak $R$-groups $f:G\to H$ is a group homomorphism
such that $f(x^r)=f(x)^r$ for all $x\in G$ and $r\in R$. The category of weak
$R$-groups and $R$-homomorphisms is denoted by $\RGr^{\mathrm{weak}}$.
\end{definition}

\begin{definition}[Weak $R$-ideal]
Let $G$ be a weak $R$-group, and let $V$ be a normal subgroup of $G$ that is also
a weak $R$-subgroup. Then $V$ is called a \emph{weak $R$-ideal} if, for all $v\in
V$, $g\in G$, and $r\in R$,
\begin{equation}
g^{-r} (gv)^r \in V.
\end{equation}
The kernel of every $R$-homomorphism of weak $R$-groups is a weak $R$-ideal.
Conversely, if $V$ is a weak $R$-ideal of $G$, then $G/V$ inherits a natural weak
$R$-group structure for which the canonical projection $G\to G/V$ is an
$R$-homomorphism~\cite[Prop.~4]{myasnikov_groups_1994}.
\end{definition}

\begin{definition}[$R$-group]
\label{def:R-group}
A weak $R$-group $G$ is called an \emph{$R$-group} (or an $R$-group in the sense
of Myasnikov--Remeslennikov) if it satisfies the following axiom:
\begin{itemize}
    \item[(R5)] $xy=yx \Rightarrow (xy)^r=x^ry^r$.
\end{itemize}
A weak $R$-subgroup of an $R$-group is called an \emph{$R$-subgroup}. The
$R$-groups form the full subcategory of $\RGr^{\mathrm{weak}}$ spanned by the
$R$-groups,
\begin{equation}
\RGr \subseteq \RGr^{\mathrm{weak}}.
\end{equation}
\end{definition}

\begin{definition}[$R$-ideal]
Let $G$ be an $R$-group. A normal subgroup $V$ of $G$ is called an
\emph{$R$-ideal} if it is an $R$-subgroup and, for all $g,h\in G$ and $r\in R$,
\begin{equation}
[g,h]\in V \Rightarrow h^{-r} g^{-r} (gh)^r \in V.
\end{equation}
Every $R$-ideal is a weak $R$-ideal, and the kernel of every $R$-homomorphism is
an $R$-ideal. If $V$ is an $R$-ideal of $G$, then $G/V$ inherits a natural
$R$-group structure for which the canonical projection $G\to G/V$ is an
$R$-homomorphism~\cite[Prop.~5]{myasnikov_groups_1994}. The intersection of any
family of $R$-ideals is again an $R$-ideal; hence, for every subset $S$ of an
$R$-group $G$, the $R$-ideal generated by $S$ is well defined.
\end{definition}

\begin{proposition}
\label{prop:Z_2(G)}
If $G$ is an $R$-group, then the center $Z(G)$ is a weak $R$-ideal, and the second
term $Z_2(G)$ of the upper central series is an $R$-subgroup.
\end{proposition}
\begin{proof}
By Lemma~\ref{lemma:centralizer}, $Z(G)$ is a weak $R$-subgroup. If $z\in Z(G)$,
$g\in G$, and $r\in R$, then (R5) gives $g^{-r}(gz)^r = z^r\in Z(G)$, so $Z(G)$ is
a weak $R$-ideal and $G/Z(G)$ is a weak $R$-group. Applying
Lemma~\ref{lemma:centralizer} again, $Z(G/Z(G))$ is a weak $R$-subgroup of
$G/Z(G)$; since $Z_2(G)$ is its preimage under $G\to G/Z(G)$, it follows that
$Z_2(G)$ is an $R$-subgroup of $G$.
\end{proof}

\begin{remark}[Universal algebra]
\label{remark:variety}
	The category of $R$-groups can be described in terms of universal algebra  (see \cite{amaglobeli2023varieties}).
	Consider the algebraic signature $\sigma_R$ consisting of the binary
	multiplication operation, the constant symbol $1$, and a unary operation
	$(-)^r$ for each $r\in R$. Then the $R$-groups form a quasi-variety of
	$\sigma_R$-algebras, defined by the monoid identities, the identities
	(R1)--(R4), and the quasi-identity (R5). That every object of this
	quasi-variety is a group follows because the unary operation corresponding
	to $-1\in R$ assigns to each element its inverse.
\end{remark}

\begin{remark}[$\RGr$ is locally presentable]
The category $\RGr$ of $R$-groups is locally finitely presentable. Indeed, a
cocomplete category with a strong generator consisting of finitely presentable
objects is locally finitely presentable~\cite[Th.~1.11]{adamek1994locally}; the
category $\RGr$ is a quasi-variety, hence cocomplete, and the free $R$-group on
one generator is a finitely presentable strong
generator~(see~\cite[\S3B]{adamek1994locally}).
\end{remark}

\begin{definition}[Free $R$-group]
The forgetful functor $\RGr \to \mathrm{Set}$ admits a left adjoint
\begin{equation}
 F_R:\mathrm{Set}\to \RGr
\end{equation}
(see~\cite[Th.~4]{myasnikov_groups_1994}). For a set $X$, the $R$-group $F_R(X)$
is called the \emph{free $R$-group on $X$}. The image of the unit $X\to F_R(X)$
generates $F_R(X)$ as an $R$-group.
\end{definition}

\begin{definition}[$R$-completion]
The forgetful functor $U_R : \RGr \to \Gr$ admits a left adjoint
\begin{equation}
		C_R : \Gr \leftrightarrows \RGr : U_R,
\end{equation}
which we call the \emph{$R$-completion functor}
(see~\cite[Th.~1]{myasnikov_groups_1994}). For every group $G$, the image of the
unit $G\to C_R(G)$ generates $C_R(G)$ as an $R$-group. Moreover,
\begin{equation}
C_R(F(X))\cong F_R(X),
\end{equation}
where $F(X)$ is the free group on $X$.
\end{definition}

\begin{remark}[$R$-completion preserves surjections]
\label{R-comp-surj}
If $\varphi:G\to H$ is a surjective group homomorphism, then
$C_R(\varphi):C_R(G)\to C_R(H)$ is a surjective $R$-homomorphism. Indeed, both
$C_R(G)$ and $C_R(H)$ are generated, as $R$-groups, by the images of $G$ and $H$
respectively, and $C_R(\varphi)$ carries the generators of $C_R(G)$ onto those of
$C_R(H)$.
\end{remark}

\begin{example}
	Regard $R$ as an $R$-group with $x^r = xr$ for all $x,r\in R$. Then $R$
	satisfies the universal property of $C_R(\mathbb{Z})$: for every $R$-group
	$H$ and every homomorphism $u : \mathbb{Z} \to U_R(H)$ there is a unique
	$R$-homomorphism $f: R \to H$ with $u = f \circ \eta_{\mathbb{Z}}$, since $f$
	must satisfy $f(r) = u(1)^r$. Therefore $C_R(\mathbb{Z}) \cong R$ is the free
	$R$-group on one generator.
\end{example}

\section{\texorpdfstring{$R$}{R}-local groups}\label{section:r-local}

In this section we introduce the class of $R$-local groups and the
$R$-localization functor. We characterize the $R$-local groups in terms of
$R$-groups and show that the $R$-localization $L_R(G)$ of a group $G$ is the
$R/1$-nullification of the $R$-completion $C_R(G)$.

\begin{definition}[Localization with respect to a morphism]
Let $C$ be a category and $f:a\to b$ be its morphism. An object $x$ is called \emph{$f$-local}
if $f$ induces a bijection
\begin{equation}
	f^* : \Hom(b,x) \overset{\cong}\longrightarrow \Hom(a,x).
\end{equation}
If $C$ is locally presentable, then the full subcategory of $f$-local objects is reflective \cite[Th.~1.39]{adamek1994locally}, and the associated
localization functor is denoted by
\begin{equation}
	L_f : C \longrightarrow C.
\end{equation}
Recall that a \emph{coaugmentation} of an endofunctor $\mathcal{F}$ is a natural
transformation $\Id \to \mathcal{F}$ from the identity functor, and a
\emph{coaugmented functor} is an endofunctor equipped with a coaugmentation. The
functor $L_f$ admits a natural coaugmentation $\eta:\Id \to L_f$ for which
$\eta_x:x \to L_f(x)$ is the universal map from $x$ to an $f$-local object. 
See~\cite[\S1C]{adamek1994locally}, \cite{casacubertaOrthogonalPairsCategories} for further details.
\end{definition}

\begin{definition}[$R$-localization]
Let $R$ be a commutative ring. We regard the unique ring homomorphism
$\bb{Z} \to R$ as a morphism in the category of groups; by abuse of notation, we
also write $R$ for its additive group. Groups local with respect to this
homomorphism are called \emph{$R$-local}, and the associated functor is denoted
by
\begin{equation}
	L_R : \Gr \longrightarrow \Gr.
\end{equation}
Since $G\cong \Hom(\bb{Z},G)$ for every group $G$, the group $G$ is $R$-local if
and only if the map
\begin{equation}
	\Hom(R,G) \to G, \qquad \varphi\mapsto \varphi(1)
\end{equation}
is a bijection; equivalently, $G$ is $R$-local if and only if, for every $x\in
G$, there is a unique homomorphism $\varphi_x : R \to G$ with $\varphi_x(1)=x$.
\end{definition}

\begin{proposition}
\label{prop:R-ring_structure_on_a_local_group}
Let $R$ be a commutative ring and let $G$ be an $R$-local group. Then:
\begin{enumerate}
    \item $G$ admits a unique $R$-group structure;
    \item for every $R$-group $H$, every homomorphism $H\to G$ is an
    $R$-homomorphism.
\end{enumerate}
\end{proposition}
\begin{proof}
	Consider the function $G\times R \to G$ given by $(x,r) \mapsto x^r :=
	\varphi_x(r)$. We check that it is an $R$-group structure.

    (R1) follows from the fact that $\varphi_x$ is a homomorphism.

    (R2) follows from $\varphi_x(1)=x$.

    (R3) Fix $r\in R$ and $x\in G$, and consider the homomorphism $R \to G$,
    $s\mapsto x^{rs}$. Since it sends $1$ to $x^r$, it coincides with
    $\varphi_{x^r}$, and hence $x^{rs} = (x^r)^s$.

    (R4) Consider the homomorphism $R\to G$, $r\mapsto y^{-1} x^r y$. Since it
    sends $1$ to $y^{-1}xy$, it coincides with $\varphi_{y^{-1}xy}$, and hence
    $y^{-1} x^r y = (y^{-1}xy)^r$.

    (R5) Suppose $x,y\in G$ commute. By (R4), $y^{-1} x^r y = x^r$, so $x^r$
    commutes with $y$ for every $r\in R$; similarly, $x^r$ commutes with $y^s$
    for all $r,s\in R$. Consider the map $R\to G$, $r\mapsto x^r y^r$. Since
    $x^{r+s}y^{r+s} = x^r x^s y^r y^s = (x^r y^r)(x^s y^s)$, it is a homomorphism
    sending $1$ to $xy$, and hence $(xy)^r = x^r y^r$.

	To prove uniqueness, let $(x,r) \mapsto \psi_x(r)$ be any $R$-group structure
	on $G$. By (R1) and (R2), each $\psi_x: R \to G$ is a homomorphism sending
	$1$ to $x$, so $\psi_x = \varphi_x$.

	Now let $f : H\to G$ be a homomorphism from an $R$-group $H$. The equality
	$f(x^r) = f(x)^r$ holds because the maps $r\mapsto f(x^r)$ and $r\mapsto
	f(x)^r$ are homomorphisms $R\to G$ that agree at $r=1$.
\end{proof}

\begin{remark}
From now on, every $R$-local group is treated as an $R$-group.
\end{remark}

\begin{definition}[$R/1$-null groups]
For a commutative ring $R$, we write
\begin{equation}
R/1 = R^+/ \langle 1_R \rangle
\end{equation}
for the quotient of the additive group $R^+$ by the subgroup generated by the
unit $1_R$.  
We say that a group $G$ is $R/1$-null if $\Hom(R/1,G)$ is trivial.  
\end{definition}

\begin{proposition}
\label{prop:r-null}
A group is $R$-local if and only if it admits an $R$-group structure and is $R/1$-null.
\end{proposition}
\begin{proof}
If $G$ is $R$-local, then it admits an $R$-group
structure (Proposition~\ref{prop:R-ring_structure_on_a_local_group}). Moreover,
every homomorphism $f: R/1 \to G$ yields, by composition with the projection
$R\to R/1$, a homomorphism $f':R\to G$ with $f'(1_R)=1$. Since $G$ is $R$-local,
the only such homomorphism is the trivial one, so $f'$ is trivial; as $R\to R/1$
is surjective, $f$ is trivial as well. Hence $G$ is $R/1$-null.

	Conversely, assume that $G$ is an $R$-group and is $R/1$-null. We must show
	that for every $x\in G$ there is a unique homomorphism $\varphi: R \to G$
	with $\varphi(1)=x$. Existence is clear: take $\varphi(r):=x^r$. For
	uniqueness, let $\varphi:R\to G$ be a homomorphism with $\varphi(1)=x$, and
	consider the map $\tau:R\to G$ defined by
	\begin{equation}
		\tau(r) = x^r \varphi(r)^{-1}.
	\end{equation}
	Since the additive group of $R$ is abelian, $\varphi(r)$ and $\varphi(s)$
	commute for all $r,s\in R$; in particular, $\varphi(r)$ commutes with $x =
	\varphi(1)$. Then (R4) implies that $\varphi(r)$ commutes with $x^s$ for all
	$r,s\in R$. The computation
	\begin{equation}
		\tau(r+s) = x^r x^s \varphi(s)^{-1}\varphi(r)^{-1} =
		x^r \varphi(r)^{-1} x^s \varphi(s)^{-1} =
		\tau(r)\tau(s)
	\end{equation}
	shows that $\tau$ is a homomorphism with $\tau(1) = 1$. Therefore $\tau$
	factors through $R/1$, and since $G$ is $R/1$-null, we obtain $\tau(r) = 1$,
	that is, $\varphi(r) = x^r$.
\end{proof}

\begin{corollary}
\label{cor:R-subgroup-local}
	If $G$ is an $R$-local group and $H$ is an $R$-subgroup of $G$, then $H$ is
	$R$-local.
\end{corollary}
\begin{proof}
	By assumption $H$ is an $R$-group, and as a subgroup of the $R/1$-null group
	$G$ it is itself $R/1$-null. By Proposition~\ref{prop:r-null}, $H$ is
	$R$-local.
\end{proof}

\begin{definition}[$A$-nullification of an $R$-group]
Let $A$ be an $R$-group. An $R$-group $G$ is called \emph{$A$-null} if $\Hom_{\RGr}(A,G)$ is
trivial. The $A$-null $R$-groups can be described as $f$-local $R$-groups in two
equivalent ways: as the groups local with respect to the map $1\to A$, or as the
groups local with respect to the map $A\to 1$. The associated localization
functor is called the \emph{$A$-nullification functor} and is denoted by
\begin{equation}
	\mathsf{N}_A : \RGr \to \RGr.
\end{equation}
The functor $\mathsf{N}_A$ can be constructed explicitly as follows. For every $R$-group $G$,
let $\mathsf{r}^1_A(G)$ be the $R$-ideal generated by the images of all
$R$-homomorphisms $A\to G$. We then define $R$-ideals $\mathsf{r}^\alpha_A(G)$ of $G$
for all ordinals $\alpha$ by transfinite recursion: $\mathsf{r}^{\alpha+1}_A(G)$ is the
preimage of $\mathsf{r}^1_A(G/\mathsf{r}^\alpha_A(G))$ under the projection $G\to G/\mathsf{r}^\alpha_A(G)$,
and for a limit ordinal $\lambda$ we set $\mathsf{r}^\lambda_A(G) = \bigcup_{\alpha<\lambda}
\mathsf{r}^\alpha_A(G)$. For sufficiently large $\alpha$ this tower stabilizes; we denote
its stable value by $\mathsf{r}_A(G)=\mathsf{r}^\alpha_A(G)$. It is then easy to see that
\begin{equation}
	\label{eq:N_r}
	\mathsf{N}_A(G) = G/\mathsf{r}_A(G).
\end{equation}
In particular, $\mathsf{N}_A$ sends epimorphisms to epimorphisms. 

This construction can be applied to the category of groups $\Gr$, if $R=\bb{Z}$. 
\end{definition}

\begin{lemma}	\label{lemma:null-closed-under-extensions}
For an  $R$-group $A$, the class of $A$-null $R$-groups is closed under extensions.
\end{lemma}
\begin{proof}
Let $1\to N\to E\to Q\to 1$ be a short exact sequence of $R$-groups with $N$ and $Q$ both
$A$-null. For every $R$-homomorphism $f:A\to E$, the composite $A\to E\to Q$ is
trivial, so $f$ factors through $N$, where it is again trivial. Thus $E$ is
$A$-null.
\end{proof}

\begin{definition}[$R/1$-nullification of an $R$-group]
Consider the $R$-group $F_R(R/1)$ and set 
\begin{equation}
N_{R/1}=\mathsf{N}_{F_R(R/1)}, \qquad r_{R/1} = \mathsf{r}_{F_R(R/1)}.
\end{equation}
Then we obtain a functor
\begin{equation}
N_{R/1}  : \RGr \to \RGr,
\qquad
N_{R/1}(G) =G/r_{R/1}(G).
\end{equation}
For any $R$-group $G$, the map $G\to N_{R/1}(G)$ is the universal $R$-homomorphism among $R$-homomorphisms $G\to H$ to an $R$-group $H$ such that $U_R(H)$ is $R/1$-null. 
\end{definition}

\begin{theorem}
\label{thm:r-local-r-groups}
	For a commutative ring $R$ and a group $G$, the homomorphism $G\to L_R(G)$ induces a surjective $R$-homomorphism
        \[C_R(G) \to L_R(G) \]
        with kernel $r_{R/1}(C_R(G))$. In particular, there is an isomorphism of coaugmented functors
        \[L_R\cong U_R \circ N_{R/1} \circ C_R.\]
\end{theorem}
\begin{proof} By Proposition~\ref{prop:r-null}, $U_R(N_{R/1}(C_R(G)))$ is $R$-local. Let $X$ be an $R$-local group.
	Write $\tilde X$ for $X$ regarded as an object of $\RGr$. Combining Proposition~\ref{prop:R-ring_structure_on_a_local_group}, Proposition~\ref{prop:r-null} and the universal property of $N_{R/1}$, we obtain
	\begin{align}
		\Hom_{\Gr}(U_R(N_{R/1}(C_R(G))), X ) & \cong \Hom_{\RGr}(N_{R/1}(C_R(G)), \tilde X) \\
		&\cong \Hom_{\RGr}(C_R(G),\tilde X) \\
		&\cong \Hom_{\Gr}(G,X).
	\end{align}
	Hence $G\to U_R(N_{R/1}(C_R(G)))$ is the $R$-localization of $G$: $L_R(G)\cong U_R(N_{R/1}(C_R(G)))$. It follows that the map $G\to L_R(G)$ induces a surjective homomorphism $C_R(G)\to L_R(G)$ with kernel $r_{R/1}(C_R(G))$. By Proposition~\ref{prop:R-ring_structure_on_a_local_group}(2), it is a surjective $R$-homomorphism.
\end{proof}

\begin{corollary}
The localization functor $L_R$ sends epimorphisms to epimorphisms.
\end{corollary}
\begin{proof}
	This follows because $U_R$, $N_{R/1}$ and $C_R$ send surjective morphisms to surjective morphisms  (Remark~\ref{R-comp-surj}).
\end{proof}

\section{Polynomial functions}
\label{section:polynomials}

The goal of this section is to provide a general method for proving polynomial
identities in nilpotent groups (Theorem~\ref{theorem:spoly-inclusion}). We will
use it later to show that the extended Hall--Petresco identities hold in
nilpotent $R$-local groups. We begin by introducing the notion of a polynomial
function, following~\cite{leibman2002polynomial} closely.

\begin{definition}[Polynomial functions]
\label{def:degree}
    Let $A$ and $G$ be groups and let $\varphi : A \to G$ be a set function. For
    $a\in A$, we define the \emph{left} and \emph{right forward differences}
    $D^\ell_a\varphi, D^r_a\varphi : A\to G$ by
    \begin{equation}
    (D^\ell_a\varphi)(x) = \varphi(ax)\,\varphi(x)^{-1},
    \qquad
    (D^r_a\varphi)(x) = \varphi(x)^{-1}\varphi(xa).
    \end{equation}
    The function $\varphi$ is called \emph{left polynomial of left degree at most
    $d$} if
    \begin{equation} D^\ell_{a_1}D^\ell_{a_2}\cdots D^\ell_{a_{d+1}}\varphi \equiv 1
    \end{equation}
    for all $a_1,\dots,a_{d+1}\in A$. By convention, the constant function
    $\varphi \equiv 1$ is left polynomial of left degree $-\infty$. \emph{Right
    polynomial functions} and their right degree are defined similarly.
\end{definition}

\begin{remark}[Polynomial functions into nilpotent groups]
If $G$ is nilpotent, then a function $\varphi:A\to G$ is left polynomial if and
only if it is right polynomial~\cite[Prop.~3.16]{leibman2002polynomial}.
In this case we call $\varphi$ simply a \emph{polynomial function}, and we denote
the set of polynomial functions by
\begin{equation}
\Poly(A,G).
\end{equation}
Moreover, $\Poly(A,G)$ is a group under pointwise
multiplication~\cite[Th.~3.2]{leibman2002polynomial}. The left and
right degrees of a polynomial function into a nilpotent group may nevertheless
differ.
\end{remark}

\begin{remark}[Polynomial functions into abelian groups]
If $G$ is abelian and $A$ is any group, then a function $\varphi:A\to G$ has left
degree at most $d$ if and only if it has right degree at most $d$ (although
$D^\ell_a \varphi$ and $D^r_a \varphi$ need not
coincide)~\cite[Cor.~2.13]{leibman2002polynomial}. In this case we
call the common value simply the \emph{degree}. Moreover, the functions of degree
at most $d$ form a subgroup
of $\Poly(A,G)$~\cite[Lemma~2.1]{leibman2002polynomial}, which we
denote by
\begin{equation}
\Poly_d(A,G)\leq \Poly(A,G).
\end{equation}
\end{remark}

\begin{remark}[Composition of a polynomial function with homomorphisms]
\label{rem:composition:poly}
Let $A$ and $B$ be groups, let $G$ and $H$ be nilpotent groups, and let
$\varphi: B \to G$ be a polynomial function of left degree at most $d$. Then, for
any homomorphisms $f:A\to B$ and $f':G\to H$, the composite $f'\circ \varphi\circ
f$ is a polynomial function of left degree at most
$d$~\cite[Prop.~1.10, 1.11]{leibman2002polynomial}. In particular, $f$ and $f'$
induce a homomorphism
\begin{equation}
(f^*,f'_*) : \Poly(B, G) \to \Poly(A,H).
\end{equation}
\end{remark}

\begin{theorem}
\label{theorem:spoly-inclusion}
    Let $f: A \to B$ be a group homomorphism, and let $G$ be a nilpotent group
    such that the map
    \begin{equation}
    f^*:\Hom(B,G) \to \Hom(A,G)
    \end{equation}
     is injective. Then the homomorphism
    \begin{equation}
    f^*:\Poly(B, G) \to \Poly(A, G)
    \end{equation}
    is a monomorphism.
\end{theorem}

\begin{lemma}
\label{lemma:poly_abelian}
Theorem~\ref{theorem:spoly-inclusion} holds for every abelian group $G$.
\end{lemma}
\begin{proof}
It suffices to prove that the homomorphism $\Poly_d(B,G)\to \Poly_d(A,G)$ is a
monomorphism for every $d\geq 0$. We argue by induction on $d$. The case $d=0$ is
clear.

Assume that $\Poly_d(B,G)\to \Poly_d(A,G)$ is a monomorphism; we show that
$\Poly_{d+1}(B,G)\to \Poly_{d+1}(A,G)$ is a monomorphism, that is, that its kernel
is trivial. Let $\varphi\in \Poly_{d+1}(B,G)$ satisfy $\varphi\circ f \equiv 1$,
and consider the function $\xi:B\times B\to G$ defined by
\begin{equation}
\xi(x,y)=\varphi(y)^{-1}\varphi(x)^{-1}\varphi(xy).
\end{equation}

We claim that $\xi$ has degree at most $d$ in each variable. For fixed $x_0 \in
B$,
\begin{align}
\xi(x_0,y) &=\varphi(x_0)^{-1} \varphi(y)^{-1}\varphi(x_0 y)\\
&= \varphi(x_0)^{-1}\, (D^\ell_{x_0}\varphi)(y),
\end{align}
so $\xi(x_0,-)$ has degree at most $d$, being the product of the constant function
$\varphi(x_0)^{-1}$ and the function $D^\ell_{x_0}\varphi$, both of degree at most
$d$. For fixed $y_0\in B$,
\begin{equation}
\xi(x,y_0) = \varphi(y_0)^{-1}\, (D^r_{y_0}\varphi)(x),
\end{equation}
and the same argument shows that $\xi(-,y_0)$ has degree at most $d$.

Next we show that $\xi\equiv 1$. If $x_0,y\in \Im(f)$, then $\varphi\circ f \equiv
1$ gives $\xi(x_0,y)=1$. Fixing $x_0\in \Im(f)$ and using the inductive hypothesis
together with $\xi(x_0,-)\in \Poly_d(B,G)$, we conclude that $\xi(x_0,y)=1$ for all
$y\in B$. Similarly, fixing $y_0\in B$ shows that $\xi(-,y_0)\equiv 1$. Therefore
$\xi\equiv 1$. Since $\varphi(1)=\varphi(f(1))=1$ and $\xi\equiv 1$, the function
$\varphi$ is a homomorphism, and the claim follows from the injectivity of
$f^*:\Hom(B,G)\to\Hom(A,G)$.
\end{proof}

\begin{proof}[Proof of Theorem~\ref{theorem:spoly-inclusion}]
We argue by induction on the nilpotency class of $G$. If $G$ is abelian, the
statement is Lemma~\ref{lemma:poly_abelian}. Assume that $\Poly(B,H)\to
\Poly(A,H)$ is a monomorphism for every nilpotent group $H$ of class $c$; we show
that $\Poly(B,G)\to \Poly(A,G)$ is a monomorphism for every nilpotent group $G$ of
class $c+1$.

We must show that the kernel of $\Poly(B,G)\to \Poly(A,G)$ is trivial. Let
$\varphi:B\to G$ be a polynomial function with $\varphi\circ f \equiv 1$. For fixed
$g\in G$, consider the polynomial function $\psi:B\to G$ given by
\begin{equation}
\psi(x) = [\varphi(x),g].
\end{equation}
Since $\psi(B)\subseteq [G,G]$, the function $\psi$ corestricts to a polynomial
function $\psi':B\to [G,G]$ with $\psi' \circ f \equiv 1$. As $[G,G]$ is nilpotent
of class at most $c$, the inductive hypothesis gives $\psi'\equiv 1$, and hence
$\psi\equiv 1$ for every $g\in G$. It follows that $\varphi(B)\subseteq Z(G)$,
where $Z(G)$ denotes the center of $G$. Thus $\varphi$ corestricts to a polynomial
function $\varphi':B\to Z(G)$ with $\varphi'\circ f\equiv 1$. By
Lemma~\ref{lemma:poly_abelian}, $\varphi'\equiv 1$, and therefore $\varphi\equiv
1$.
\end{proof}

\begin{corollary}
\label{cor:spoly-partial}
    If $G$ is a nilpotent $R$-local group, then the map
    \[ \Poly(R, G) \to \Poly(\bb{Z}, G)\]
    is injective.
\end{corollary}

\section{\texorpdfstring{$\lambda$}{lambda}-rings and binomial rings}
In this section we recall the definitions of a $\lambda$-ring and a normed
pre-$\lambda$-ring, and their relationship to binomial rings.

\begin{definition}[Pre-$\lambda$-ring]
A \emph{pre-$\lambda$-ring} is a commutative ring $R$ equipped with functions
$\lambda^n: R \to R$, $n \geq 0$, called \emph{$\lambda$-operations}, such that
for all $r, s \in R$ the following axioms hold:
	\begin{itemize}
		\item[($\lambda1$)] \label{lambda-zero-axiom} $\lambda^0(r) = 1$,
		\item[($\lambda2$)] \label{lambda-one-axiom} $\lambda^1(r) = r$,
		\item[($\lambda3$)] \label{lambda-sum-axiom} $\lambda^n (r + s) = \sum_{i + j = n} \lambda^i(r)\, \lambda^j(s)$.
	\end{itemize}
\end{definition}

\begin{notation}
For a commutative ring $R$, we write
\begin{equation}
W(R)=(1+tR[[t]])^\times
\end{equation}
for the group, under multiplication, of power series whose constant term is $1$.
If $R$ is a pre-$\lambda$-ring, then each $x\in R$ determines a power series
$\lambda_t(x)=\sum_{n\geq 0} \lambda^n(x)\,t^n$, and ($\lambda3$) is equivalent to
the statement that
\begin{equation}
 \lambda_t:R^+\to W(R)
\end{equation}
is a group homomorphism.
\end{notation}

\begin{proposition}
\label{prop:lambda_is_poly}
Let $R$ be a pre-$\lambda$-ring. Then, for every $n\geq 1$, the map
$\lambda^n : R \to R$ is a polynomial function of degree at most $n$.
\end{proposition}
\begin{proof}
We argue by induction on $n$. The base cases $n=0$ and $n=1$ follow from
($\lambda1$) and ($\lambda2$): $\lambda^0$ is constant and $\lambda^1=\mathrm{id}$
is a homomorphism, hence of degree at most $1$. For the inductive step, axioms
($\lambda1$) and ($\lambda3$) give, for fixed $s$,
\begin{equation}
(D_s \lambda^n)(r) = \sum_{i = 0}^{n-1} \lambda^{n-i}(s)\, \lambda^{i}(r).
\end{equation}
By the inductive hypothesis $\lambda^i$ has degree at most $i$, and since
$\lambda^{n-i}(s)$ is a constant, the map $r\mapsto \lambda^{n-i}(s)\, \lambda^{i}(r)$
also has degree at most $i$. Hence $D_s \lambda^n$ has degree at most $n-1$, and
therefore $\lambda^n$ has degree at most $n$.
\end{proof}

\begin{definition}[Normed pre-$\lambda$-ring]
A pre-$\lambda$-ring is called \emph{normed} if it satisfies the following axiom:
\begin{itemize}
    \item[($\lambda4$)] $\lambda^n(1) = 0$ for $n \geq 2$.
\end{itemize}
\end{definition}

\begin{definition}[$\lambda$-ring]
A normed pre-$\lambda$-ring is called a \emph{$\lambda$-ring} if it satisfies the
following axioms:
\begin{itemize}
\item[($\lambda5$)] \label{lambda-product-axiom} $\lambda^n(rs) = P_n (\lambda^1(r), \dotsc, \lambda^n(r);\, \lambda^1(s), \dotsc, \lambda^n(s))$,
\item[($\lambda6$)] \label{lambda-composition-axiom} $\lambda^n(\lambda^m(r)) = P_{n,m} (\lambda^1(r), \dotsc, \lambda^{nm}(r))$,
\end{itemize}
where $P_n$ and $P_{n,m}$ are universal polynomials with integer coefficients
(see~\cite[Examples~1.7 and~1.9]{lambda-rings}).
\end{definition}

\begin{definition}[Binomial ring]
A \emph{binomial ring} is a $\bb{Z}$-torsion-free ring $R$ such that the binomial
coefficient
\begin{equation}
\binom{r}{n} = \frac{r(r-1)\cdots (r-(n-1))}{n!} \in R \otimes \bb{Q}
\end{equation}
lies in $R$ for all $r\in R$ and all $n\geq 1$.
\end{definition}

\begin{definition}[Adams operations]
Let $R$ be a $\lambda$-ring. For every $r\in R$, the power series
$\lambda_t(r)\in R[[t]]$ is invertible, since its constant term is
$\lambda^0(r)=1$. Consider the power series $\psi_{t}(r)\in R[[t]]$ defined by
\begin{equation}
\psi_{-t}(r) = (-t)\, \lambda_t(r)'\, \lambda_t(r)^{-1},
\end{equation}
where $\lambda_t(r)'$ denotes the derivative with respect to $t$. The coefficients
of $\psi_t(r)$ are denoted by $\psi^n(r)\in R$, and the maps
\begin{equation}
\psi^n : R\to R, \quad n\geq 0,
\end{equation}
are called the \emph{Adams operations}. The Adams operations are ring
homomorphisms~\cite[Th.~3.6]{lambda-rings}.
\end{definition}

\begin{theorem}[{\cite[Th.~5.3]{lambda-rings}}]
\label{th:binomial-lambda}
A binomial ring carries a unique $\lambda$-ring structure in which all Adams
operations are the identity. Conversely, a $\lambda$-ring in which all Adams
operations are the identity is a binomial ring. In either case the
$\lambda$-operations are given by $\lambda^n(r)=\binom{r}{n}$.
\end{theorem}

\section{Hall--Petresco \texorpdfstring{$R$}{R}-groups}
\label{subsection:Hall--Petresco}

In this section we recall the Hall--Petresco identity
(see~\cite{clement_theory_2017, warfield2006nilpotent}). Then, for a normed
pre-$\lambda$-ring $R$, we define the notion of a Hall--Petresco $R$-group and
show that a Hall--Petresco $R$-group structure on a group $G$ induces one on the
quotient $G/Z(G)$ by its center. Throughout this section, $R$ denotes a normed
pre-$\lambda$-ring.

\begin{definition}[Hall--Petresco words]
\label{def:hall-petresco-words}
Let $F=F(x_1,\dots,x_m)$ be the free group on $m$ generators, with lower central
series $\gamma_k(F)$, $k\geq 1$. The $k$-th Hall--Petresco word $\tau_k\in F$ is
defined recursively by $\tau_1=x_1\cdots x_m$ and the relation
\begin{equation}
x_1^k\cdots x_m^k = \tau_1^k\, \tau_2^{\binom{k}{2}} \cdots \tau_{k-1}^{\binom{k}{k-1}}\, \tau_k.
\end{equation}
The key property of the Hall--Petresco words is that $\tau_k\in
\gamma_k(F)$~\cite[Cor.~4.1]{clement_theory_2017}. The \emph{Hall--Petresco
identity} states that for every nilpotent group $G$ of class $n$, all elements
$g_1,\dots,g_m\in G$, and every $k\in \mathbb{Z}$,
\begin{equation}
g_1^k\cdots g_m^k = \tau_1(\bar g)^k\, \tau_2(\bar g)^{\binom{k}{2}} \cdots \tau_{n-1}(\bar g)^{\binom{k}{n-1}}\, \tau_{n}(\bar g)^{\binom{k}{n}},
\end{equation}
where $\bar g=(g_1,\dots,g_m)$ and $\tau_k(\bar g)\in \gamma_k(G)$ is obtained
from $\tau_k$ by the substitution $x_i\mapsto g_i$.
\end{definition}

\begin{definition}[Hall--Petresco $R$-group]
\label{def:hall-petreco}
Let $R$ be a normed pre-$\lambda$-ring. An $R$-group $G$ is called a
\emph{Hall--Petresco $R$-group} if, for all elements $g_1,\dots,g_m\in G$ for
which the $R$-subgroup $\langle g_1,\dots,g_m \rangle_R$ is nilpotent of class at
most $n$, the identity
\begin{equation}
\label{eq:HP-lambda}
g_1^r\cdots g_m^r = \tau_1(\bar g)^{\lambda^1(r)}\, \tau_2(\bar g)^{\lambda^2(r)} \cdots \tau_{n-1}(\bar g)^{\lambda^{n-1}(r)}\, \tau_{n}(\bar g)^{\lambda^n(r)}
\end{equation}
holds for all $r\in R$. We call \eqref{eq:HP-lambda} the \emph{Hall--Petresco
identity for the $R$-powers of $g_1,\dots,g_m$}. The full subcategory of $\RGr$
spanned by the Hall--Petresco $R$-groups is denoted by
\begin{equation}
\RGr^{\mathrm{HP}} \subseteq \RGr.
\end{equation}
\end{definition}

\begin{remark}[Equivalent definition of a Hall--Petresco $R$-group]
Later, as a corollary of one of our main results
(Corollary~\ref{cor:<g_1,..,g_n>-nilpotent}), we will show the following: for a
Hall--Petresco $R$-group $G$, elements $g_1,\dots,g_m\in G$, and $n\geq 0$, the
$R$-subgroup $\langle g_1,\dots,g_m \rangle_R$ is nilpotent of class at most $n$
if and only if the ordinary subgroup $\langle g_1,\dots,g_m \rangle$ is nilpotent
of class at most $n$. Consequently, in the definition of a Hall--Petresco
$R$-group, one may replace $\langle g_1,\dots,g_m \rangle_R$ by
$\langle g_1,\dots,g_m \rangle$.
\end{remark}

\begin{remark}[$R$-powered nilpotent groups in the sense of Hall]
\label{rem:R-powered}
If $R$ is a binomial ring, then a nilpotent Hall--Petresco $R$-group is precisely
an $R$-powered nilpotent group in the sense of Hall~\cite{nilpotent-groups}.
\end{remark}

\begin{proposition}
\label{prop:quotient of Hall--Petresco by Z}
If $G$ is a Hall--Petresco $R$-group, then its center $Z(G)$ is an $R$-ideal, and
the quotient $G/Z(G)$ is again a Hall--Petresco $R$-group.
\end{proposition}
\begin{proof}
Proposition \ref{prop:Z_2(G)} implies that the center $Z(G)$ is a weak $R$-ideal. Then $G/Z(G)$ is a weak $R$-group. Let us prove that $Z(G)$ is an $R$-ideal. Take
$g,h\in G$ with $[g,h]\in Z(G)$, and let $r\in R$. Denote by $\bar g,\bar h$ their images in $G/Z(G)$. Since $\bar g$ commutes with $\bar h$,  Proposition \ref{prop:commutative} implies that the weak $R$-subgroup $\langle \bar g, \bar h \rangle_R$ of $G/Z(G)$ is commutative. Therefore the $R$-subgroup $\langle g,h \rangle_R$ of $G$ is nilpotent of class at most $2$. The Hall--Petresco identity for
the $R$-powers of $g,h$ gives
\begin{equation}
g^r h^r = (gh)^r\, [h,g]^{\lambda^2(r)},
\end{equation}
and therefore
\begin{equation}
h^{-r} g^{-r} (gh)^r = [h,g]^{-\lambda^2(r)} \in Z(G).
\end{equation}
Thus $Z(G)$ is an $R$-ideal.

We now show that $G/Z(G)$ is a Hall--Petresco $R$-group. Let $g'_1, \dots, g'_m
\in G/Z(G)$ lie in a nilpotent $R$-subgroup $H'\leq G/Z(G)$ of class at most $n$.
Let $H\subseteq G$ be the preimage of $H'$, and choose preimages $g_i\in H$ of the
$g'_i$; then $H$ is an $R$-subgroup of $G$. Since $H\cap Z(G)$ is central in $H$
and $H/(H\cap Z(G))\cong H'$, the group $H$ is nilpotent of class at most $n+1$,
and $\gamma_{n+1}(H)\subseteq H\cap Z(G)$. In particular $\tau_{n+1}(\bar g)\in
H\cap Z(G)$, and since $H\cap Z(G)$ is an $R$-subgroup,
$\tau_{n+1}(\bar g)^{\lambda^{n+1}(r)}\in Z(G)$. Projecting the Hall--Petresco
identity for the $R$-powers of $g_1,\dots,g_m$ along $G\to G/Z(G)$, the final
factor $\tau_{n+1}(\bar g)^{\lambda^{n+1}(r)}$ maps to the identity, and we obtain
the Hall--Petresco identity for the $R$-powers of $g'_1,\dots,g'_m$.
\end{proof}

\begin{remark}[Hall--Petresco $R$-groups form a quasi-variety]
	In the language of universal algebra, the Hall--Petresco $R$-groups form a
	quasi-subvariety of the quasi-variety of $R$-groups. Indeed, the condition
	that $g_1,\dots,g_m\in G$ are contained in a nilpotent $R$-subgroup of class
	at most $c$ can be expressed as the vanishing of the $(c+1)$-fold commutators
	of the elements obtained from $g_1,\dots,g_m$ by the $R$-group operations.
\end{remark}

\begin{remark}[Hall--Petresco $R$-groups form a reflective subcategory]
\label{rem:HP-reflective}
For an $R$-group $G$, there is a universal $R$-homomorphism $G\to \mathrm{HP}(G)$
to a Hall--Petresco $R$-group: the group $\mathrm{HP}(G)$ is a quotient of $G$
obtained by a transfinite construction in which each step passes to the quotient
of the previous one by all Hall--Petresco identities. Hence $\RGr^{\mathrm{HP}}$
is a reflective subcategory of $\RGr$.
\end{remark}

\begin{definition}[Hall--Petresco $R$-completion]
\label{def:HP-adjunction}
For a group $G$, we write $C_R^{\mathrm{HP}}(G)$ for the Hall--Petresco $R$-group
$\mathrm{HP}(C_R(G))$ obtained from the $R$-completion $C_R(G)$ of $G$. This yields an adjunction
\begin{equation}
C_R^{\mathrm{HP}}: \Gr \leftrightarrows \RGr^{\mathrm{HP}} : U_R^{\mathrm{HP}},
\end{equation}
where $U_R^{\mathrm{HP}}$ is the forgetful functor.
\end{definition}

\section{Nilpotency preservation}

In this section we prove, for a normed pre-$\lambda$-ring $R$, that if $G$ is a
nilpotent group then the associated Hall--Petresco $R$-group
$C^{\mathrm{HP}}_R(G)$ is again nilpotent. As a corollary we obtain our main
result: the $R$-localization of a nilpotent group is nilpotent. Throughout this
section, $R$ denotes a normed pre-$\lambda$-ring.

\begin{proposition}
\label{prop:local are Hall--Petresco}
Every $R$-local group is a Hall--Petresco $R$-group.
\end{proposition}
\begin{proof}
Let $G$ be an $R$-local group, and let $g_1,\dots,g_m\in H$, where $H\leq G$ is a
nilpotent $R$-subgroup. Then $H$ is also $R$-local, so
Corollary~\ref{cor:spoly-partial} shows that the map
\begin{equation}
\label{eq:H-poly}
\Poly(R,H) \to \Poly(\bb{Z},H)
\end{equation}
is injective. By Proposition~\ref{prop:lambda_is_poly}, each $\lambda^n:R\to R$ is
polynomial, and the composite of a polynomial function with a homomorphism is
again polynomial (Remark~\ref{rem:composition:poly}); hence, for every $n$, the
map $R\to H$, $r\mapsto \tau_n(\bar g)^{\lambda^n(r)}$, is polynomial. Since a
product of polynomial functions into a nilpotent group is polynomial, both sides
of \eqref{eq:HP-lambda} are polynomial maps $R\to H$. They agree on
$\bb{Z}\subseteq R$ by the classical Hall--Petresco identities, so by injectivity
of \eqref{eq:H-poly} they agree on all of $R$. This is precisely the
Hall--Petresco identity for the $R$-powers of $g_1,\dots,g_m$.
\end{proof}

\begin{corollary}
\label{cor:L_R is a quotient of HP}
For a group $G$, the map $G\to L_R(G)$ induces a surjective $R$-homomorphism
\begin{equation}
C^{\mathrm{HP}}_R(G) \to L_R(G),
\end{equation}
with kernel $r_{R/1}(C^{\mathrm{HP}}_R(G))$.
\end{corollary}
\begin{proof}
This follows from Theorem~\ref{thm:r-local-r-groups} and
Proposition~\ref{prop:local are Hall--Petresco}.
\end{proof}

\begin{lemma}
\label{lemma:center_to_center}
For a group $G$, the canonical map $\eta : G \to C^{\mathrm{HP}}_R(G)$ sends the
center into the center:
\begin{equation}
\eta(Z(G)) \subseteq Z(C^{\mathrm{HP}}_R(G)).
\end{equation}
\end{lemma}
\begin{proof}
Let $z\in Z(G)$. Axiom (R4) implies that the centralizer $\mathsf{Cent}(\eta(z))$ in
$C^{\mathrm{HP}}_R(G)$ is an $R$-subgroup. Since it contains $\Im(\eta)$, which
generates $C^{\mathrm{HP}}_R(G)$ as an $R$-group, we have
$\mathsf{Cent}(\eta(z))=C^{\mathrm{HP}}_R(G)$. Therefore $\eta(z)\in Z(C^{\mathrm{HP}}_R(G))$.
\end{proof}

\begin{proposition}
\label{prop:F^{(n-1)}(G) is nilpotent} Hall-Petresco $R$-completion of a nilpotent group of class at most $n$ is a nilpotent group of class at most $n$.
\end{proposition}
\begin{proof}
Recall that $\eta: G\to C^{\mathrm{HP}}_R(G)$ is the universal
homomorphism to a Hall--Petresco $R$-group: every homomorphism $G\to A$ to a
Hall--Petresco $R$-group lifts uniquely to an $R$-homomorphism
$C^{\mathrm{HP}}_R(G)\to A$.

We argue by induction on $n$. If $n = 0$ then $G=1$, and the statement is clear.
Suppose the claim holds for nilpotent groups of class at most $n$, and let $G$ be
nilpotent of class at most $n+1$. Then $G/Z(G)$ is nilpotent of class at most $n$, so
by the inductive hypothesis $C^{\mathrm{HP}}_R(G/Z(G))$ is nilpotent of class at most
$n$. By Lemma~\ref{lemma:center_to_center}, the map $\eta$ induces a homomorphism
$G/Z(G) \to C^{\mathrm{HP}}_R(G)/Z(C^{\mathrm{HP}}_R(G))$. By
Proposition~\ref{prop:quotient of Hall--Petresco by Z}, the quotient
$C^{\mathrm{HP}}_R(G)/Z(C^{\mathrm{HP}}_R(G))$ is a Hall--Petresco $R$-group, so
this homomorphism lifts to an $R$-homomorphism
\begin{equation}
\label{eq:map-G/Z}
C^{\mathrm{HP}}_R(G/Z(G))\to C^{\mathrm{HP}}_R(G)/Z(C^{\mathrm{HP}}_R(G)).
\end{equation}
Both $R$-groups are generated by the image of $G$, so \eqref{eq:map-G/Z} is
surjective. Hence $C^{\mathrm{HP}}_R(G)/Z(C^{\mathrm{HP}}_R(G))$ is nilpotent of
class at most $n$, and therefore $C^{\mathrm{HP}}_R(G)$ is nilpotent of class at
most $n+1$.
\end{proof}

\begin{corollary}
\label{cor:<g_1,..,g_n>-nilpotent}
Let $G$ be a Hall--Petresco $R$-group, let $g_1,\dots,g_m\in G$, and let $n\geq
0$. Then the $R$-subgroup $\langle g_1,\dots,g_m \rangle_R$ is nilpotent of class
at most $n$ if and only if the ordinary subgroup $\langle g_1,\dots,g_m \rangle$
is nilpotent of class at most $n$.
\end{corollary}
\begin{proof}
If $\langle g_1,\dots,g_m \rangle_R$ is nilpotent of class at most $n$, then so is
its subgroup $\langle g_1,\dots,g_m \rangle$. Conversely, suppose $\langle
g_1,\dots,g_m \rangle$ is nilpotent of class at most $n$. Since $G$ is a
Hall--Petresco $R$-group, the inclusion $\langle g_1,\dots,g_m \rangle \hookrightarrow
G$ induces an $R$-homomorphism $C_R^{\mathrm{HP}}(\langle g_1,\dots,g_m \rangle)\to
G$ whose image is $\langle g_1,\dots,g_m \rangle_R$. By
Proposition~\ref{prop:F^{(n-1)}(G) is nilpotent},
$C_R^{\mathrm{HP}}(\langle g_1,\dots,g_m \rangle)$ is nilpotent of class at most
$n$, and hence so is its quotient $\langle g_1,\dots,g_m \rangle_R$.
\end{proof}

\begin{theorem}
	\label{thm:nilpotency-preservation}
	For a normed pre-$\lambda$-ring $R$, the $R$-localization of a nilpotent
	group is nilpotent.
\end{theorem}
\begin{proof}
This follows from Proposition~\ref{prop:F^{(n-1)}(G) is nilpotent} and
Corollary~\ref{cor:L_R is a quotient of HP}.
\end{proof}

\begin{remark}[$R$-localizations of finite $p$-groups]
 In 90's, Farjoun posed another, similar question: is it true that for
every localization $L$ on the category of groups and every finite $p$-group $G$, the coaugmentation $G\to
L(G)$ is an epimorphism? For $R$-localizations the answer is affirmative: in her
(unpublished) thesis, Gemma Bastardas proved that $L_f$ has this property whenever
$f:A\to B$ is a homomorphism of abelian
groups~\cite[Th.~B, p.~27]{bastardas2004localitzacions}.
\end{remark}

\section{\texorpdfstring{$E$}{E}-rings}

In this section, for a commutative ring $R$, we define a quotient ring $R^e$ that
is an $E$-ring. We show that the $R$-localization is the $R/1$-nullification of
the $R^e$-localization; consequently, many questions about $R$-localizations
reduce to the case of $E$-rings. We also recall the notion of a solid ring,
introduced by Bousfield--Kan, and show that if $R$ is an $E$-ring, then $L_R$ is
weakly right exact only if $R$ is solid.

\begin{definition}[$E$-ring]
	A commutative ring $R$ is called an \emph{$E$-ring} if the additive group
	$R^+$ is $R$-local.
\end{definition}

\begin{remark}[Equivalent definition of an $E$-ring]
A commutative ring $R$ is an $E$-ring if and only if the evaluation map
\begin{equation}
\mathrm{ev}_1 : \End_{\bb{Z}}(R^+) \to R, \qquad \varphi\mapsto \varphi(1),
\end{equation}
is a ring isomorphism. (This explains the ``$E$'' in the name: it stands for
\emph{endomorphism}.) Equivalently, $R$ is an $E$-ring if and only if every group
endomorphism of $R^+$ is multiplication by an element of $R$. Note that the ring
structure on an $E$-ring $R$ is determined by the abelian group $R^+$ together
with the choice of unit $1\in R$.
\end{remark}

\begin{remark}[Uniqueness of the ring structure on an $E$-ring]
\label{rem_uniquness}
If an abelian group $R^+$ admits an $E$-ring structure with unit $1\in R^+$, then
a ring structure on $R^+$ with this unit is unique. Indeed, a ring structure is
determined by a homomorphism $R\to \End_{\bb{Z}}(R^+)$, which must be the inverse
of $\mathrm{ev}_{1} : \End_{\bb{Z}}(R^+)\to R$.
\end{remark}

\begin{definition}[The $E$-ring of a ring]
\label{def:R^e}
By~\cite[Th.~3.7]{casacuberta2000structures}, for every localization $L$ on the
category of groups, $L(\bb{Z})$ carries a unique $E$-ring structure for which the
coaugmentation $\bb{Z}\to L(\bb{Z})$ is a ring homomorphism. For a commutative
ring $R$, we denote the $E$-ring associated with the $R$-localization by
\begin{equation}
R^e = L_R(\bb{Z}).
\end{equation}
Since $R^e$ is $R$-local, the map $\bb{Z} \to R^e$ lifts uniquely to a map
\begin{equation}
q_R : R \to R^e.
\end{equation}
\end{definition}

\begin{proposition}
\label{prop:R->R^e-surjective}
The map $q_R:R\to R^e$ is a surjective ring homomorphism with kernel
$r_{R/1}(R^+)$.
\end{proposition}
\begin{proof}
By Theorem~\ref{thm:r-local-r-groups} and the identification $C_R(\bb{Z})\cong R$,
we have $R^e\cong N_{R/1}(R^+)$, and hence $R^e\cong R^+/r_{R/1}(R^+)$
(see~\eqref{eq:N_r}), with $q_R$ the canonical projection. A straightforward
transfinite induction shows that $r_{R/1}^\alpha(R^+)$ is an ideal of $R$ for
every ordinal $\alpha$; therefore $r_{R/1}(R^+)$ is an ideal of $R$, and $R^e$ is
a quotient ring of $R$. By Remark~\ref{rem_uniquness}, this quotient-ring
structure on $R^e$ coincides with its $E$-ring structure.
\end{proof}

\begin{corollary}
\label{cor:Re-null}
The abelian group $(R^e)^+$ is the $R/1$-nullification of $R^+$.
\end{corollary}

\begin{remark}
If $R$ is an $E$-ring, then $q_R$ is an isomorphism $R\cong R^e$. Indeed, in this
case $R^+$ is $R$-local, hence $R/1$-null, so $r_{R/1}(R^+)=0$. It follows that
for every commutative ring $R$,
\begin{equation}
L_{R^e}(\bb{Z})\cong R^e.
\end{equation}
\end{remark}

\begin{proposition}
\label{prop:epi-L_E->L_R}
A group is $R$-local if and only if it is both $R^e$-local and $R/1$-null.
Moreover, for every group $G$, the map $G\to L_R(G)$ induces an epimorphism
\begin{equation}
L_{R^e}(G)\to L_R(G)
\end{equation}
with kernel $r_{R/1}(L_{R^e}(G))$. In other words, there is an isomorphism of $R$-groups
\begin{equation}
L_R(G) \cong N_{R/1}(L_{R^e}(G)).
\end{equation}
\end{proposition}
\begin{proof}
Let $G$ be an $R$-local group. Since $\bb{Z} \to R^e$ is the universal map to an
$R$-local group, the map $\Hom(R^e,G)\to \Hom(\bb{Z},G)$ is a bijection; thus $G$
is $R^e$-local, and it is $R/1$-null by Theorem~\ref{thm:r-local-r-groups}.

Conversely, assume that $G$ is $R^e$-local and $R/1$-null. Then $G$ admits an
$R^e$-group structure, and the homomorphism $q_R:R\to R^e$ pulls it back to an
$R$-group structure on $G$. By Theorem~\ref{thm:r-local-r-groups}, $G$ is
$R$-local.

For the functorial statement, let $X$ be an $R$-local group, $\tilde X$ be the group $X$ treated as an $R$-group and $G$ be any group. Note that $N_{R/1}(L_{R^e}(G))$ is an $R/1$-null $R$-group, and hence, it is $R$-local. Then using Proposition \ref{prop:R-ring_structure_on_a_local_group}(2), we obtain  
\begin{align}
  \Hom_{\Gr}(U_R(N_{R/1}(L_{R^e}(G))),X) &= \Hom_{\RGr}(N_{R/1}(L_{R^e}(G)),\tilde X)
  \\
  & \cong \Hom_{\RGr}(L_{R^e}(G),\tilde X)\\
  & \cong \Hom_{R^e\Gr}(L_{R^e}(G),\tilde X)
  \\  
  & \cong \Hom_{\Gr}(L_{R^e}(G),X)
  \\  
  & \cong \Hom_{\Gr}(G,X).
\end{align}
Therefore, $G\to U_R(N_{R/1}(L_{R^e}(G)))$ is the $R$-localization: $L_R(G)\cong U_R(N_{R/1}(L_{R^e}(G)))$.
\end{proof}

\begin{remark}
\label{rem:L_R(R)}
Since every $R$-local group $G$ is also $R^e$-local, the map $\Hom(R^e,G)\to
\Hom(R,G)$ is a bijection. Therefore $q_R$ induces an isomorphism
\begin{equation}
L_R(R)\cong R^e.
\end{equation}
\end{remark}

\begin{lemma}
\label{lemma:localization of free abelian}
For an $E$-ring $R$ and a free abelian group $A$, there is a natural isomorphism
\begin{equation}
L_R(A) \cong A\otimes R.
\end{equation}
\end{lemma}
\begin{proof}
Let $X$ be a set; we show that $\bb{Z}^{\oplus X}\to R^{\oplus X}$ is the
$R$-localization. Being reflective, the subcategory of $R$-local groups is closed
under limits, and in particular under products. Since $R$ is an $E$-ring, the
additive group $R^+$ is $R$-local, so the product $R^X$ is $R$-local as well.
Finally, $R^{\oplus X}$ is an $R$-subgroup of the $R$-local group $R^X$, and hence
is itself $R$-local (Corollary \ref{cor:R-subgroup-local}).

Because $L_R$ takes abelian groups to abelian groups, it suffices to check that
$\bb{Z}^{\oplus X}\to R^{\oplus X}$ is a universal map to an abelian $R$-local
group. Let $G$ be an $R$-local abelian group. Then
\begin{equation}
\Hom(R^{\oplus X}, G) \cong \Hom(R,G)^{X} \cong \Hom(\bb{Z},G)^X \cong \Hom(\bb{Z}^{\oplus X},G),
\end{equation}
where the middle isomorphism uses that $G$ is $R$-local. Therefore
$\bb{Z}^{\oplus X}\to R^{\oplus X}$ is the $R$-localization.
\end{proof}

\begin{definition}[Solid ring]
A commutative ring $R$ is called \emph{solid} if the multiplication map
$R\otimes_{\bb{Z}} R \to R$ is an isomorphism. For example, $\bb{Z}[P^{-1}]$ is
solid, whereas $\bb{Z}_p$ is not.
\end{definition}

\begin{lemma}
\label{lemma:solid-is-E}
Every solid ring is an $E$-ring.
\end{lemma}
\begin{proof}
Let $f:R^+\to R^+$ be a group homomorphism. Then $R\otimes f: R\otimes R^+\to
R\otimes R^+$ is an $R$-module homomorphism, and under the isomorphism $R\otimes
R^+\cong R^+$ (which holds since $R$ is solid) it is identified with $f$. Hence
$f$ is an $R$-module homomorphism, so $f(r)=f(1)\, r$ for all $r\in R$; that is,
$f$ is multiplication by $f(1)$.
\end{proof}

\begin{definition}[Weakly right exact functor]
Let $C$ and $D$ be categories admitting reflexive coequalizers. A functor
$F:C\to D$ is called \emph{weakly right exact} if it preserves reflexive
coequalizers. Many equivalent characterizations of the weakly right exact
endofunctors of the category of groups can be found
in~\cite[Th.~2.5]{right-exact}.
\end{definition}

\begin{proposition}
\label{prop:weakly right and solid}
Let $R$ be an $E$-ring. If $L_R$ is weakly right exact, then $R$ is solid.
\end{proposition}
\begin{proof}
Choose an exact sequence $A\to B \to R \to 0$ with $A$ and $B$ free abelian. Since
$L_R$ is weakly right exact, the sequence $L_R(A) \to L_R(B) \to L_R(R)\to 0$ is
again exact~\cite[Lemma~3.2]{right-exact}. By
Lemma~\ref{lemma:localization of free abelian}, $L_R(A)\cong A\otimes R$ and
$L_R(B)\cong B\otimes R$, and by Remark~\ref{rem:L_R(R)}, $L_R(R)\cong R$.
Therefore the sequence $A\otimes R\to B\otimes R\to R\to 0$ is exact. Since
$-\otimes R$ is right exact, applying it to $A\to B\to R\to 0$ shows that the
cokernel of $A\otimes R\to B\otimes R$ is $R\otimes R$; comparing with the
sequence above gives $R\otimes R\cong R$. Thus $R$ is solid.
\end{proof}

\begin{corollary}
The $\bb{Z}_p$-localization is not weakly right exact.
\end{corollary}

\section{\texorpdfstring{$\lambda$}{lambda}-ring structures on \texorpdfstring{$E$}{E}-rings}

In this section we show that an $E$-ring admits a normed pre-$\lambda$-ring
structure if and only if it is binomial. We also show that if a commutative ring
$R$ admits a normed pre-$\lambda$-ring structure, then $R^e$ is binomial, and we
give an example of an $E$-ring that is not binomial.

\begin{proposition}
\label{prop:e-ring-pre-lambda-is-lambda}
Let $R$ be an $E$-ring. Then the following are equivalent:
\begin{enumerate}
    \item $R$ admits a normed pre-$\lambda$-ring structure;
    \item $R$ is binomial;
    \item $R$ admits a unique $\lambda$-ring structure.
\end{enumerate}
\end{proposition}
\begin{proof}
First we show that any normed pre-$\lambda$-ring structure on $R$ is automatically
a $\lambda$-ring structure; that is, it satisfies the polynomial identities
($\lambda5$) and ($\lambda6$). Since $R$ is an $E$-ring and
$L_R(\bb{Z}^m)\cong R^m$ for every $m\geq 1$
(Lemma~\ref{lemma:localization of free abelian}),
Theorem~\ref{theorem:spoly-inclusion} shows that the map
$\Poly(R^m,R) \to \Poly(\mathbb{Z}^m,R)$ is injective; it therefore suffices to
verify these identities for integer arguments. Axioms ($\lambda1$)--($\lambda4$)
give $\lambda^n(r+1) = \lambda^n(r) + \lambda^{n-1}(r)$ for all $r \in R$, from
which it follows by induction that $\lambda^n(k\cdot 1_R) = \binom{k}{n}\, 1_R$
for all $k \in \mathbb{Z}$. Thus, on integer arguments, both sides of ($\lambda5$)
and ($\lambda6$) reduce to the corresponding identities for binomial
coefficients, which hold. Hence every normed pre-$\lambda$-ring structure on $R$
is a $\lambda$-ring structure. In particular, (1) implies (3); the converse is
trivial, since every $\lambda$-ring is by definition a normed pre-$\lambda$-ring.

Now assume that $R$ is a $\lambda$-ring. Since the Adams operations $\psi^n:R\to
R$ are additive and $R$ is an $E$-ring, we have $\psi^n(r) = \psi^n(1)\, r = r$
for all $r \in R$. By Theorem~\ref{th:binomial-lambda}, $R$ is therefore binomial,
and any $\lambda$-ring structure on $R$ is given by the binomial coefficients
$\lambda^n(r) = \binom{r}{n}$. This proves that (3) implies (2). Finally, a
binomial ring carries a $\lambda$-ring structure
(Theorem~\ref{th:binomial-lambda}), which is in particular a normed
pre-$\lambda$-ring structure, so (2) implies (1).
\end{proof}

\begin{lemma}
\label{lemma:w-null}
Let $R$ be a $E$-ring. Then $W(R)$ is $R/1$-null.
\end{lemma}
\begin{proof}
	For $n\geq 1$, set $U_n=1+t^nR[[t]]$. There is a short exact sequence
	\begin{equation}
	0\to R^+ \to W(R)/U_{n+1}\to W(R)/U_n\to 0
	\end{equation}
	and an isomorphism $W(R)\cong\varprojlim_n W(R)/U_n$. The statement therefore
	follows from the fact that the class of $R/1$-null groups is closed under extensions and limits.
\end{proof}

\begin{proposition}
\label{prop:lambda-on-R^e}
Any normed pre-$\lambda$-ring structure on $R$ induces a $\lambda$-ring structure
on $R^e$ for which $q_R:R\to R^e$ is a morphism of pre-$\lambda$-rings.
\end{proposition}
\begin{proof}
The group $(R^e)^+$ is $R/1$-null by Theorem~\ref{thm:r-local-r-groups}, so
$W(R^e)$ is $R/1$-null by Lemma~\ref{lemma:w-null}. Therefore the composite
\begin{equation}
R^+ \tox{\lambda_t} W(R) \tox{W(q_R)} W(R^e)
\end{equation}
factors uniquely through the $R/1$-nullification of $R^+$, which is $(R^e)^+$
(Corollary~\ref{cor:Re-null}). This yields a homomorphism $\lambda_t^e : (R^e)^+
\to W(R^e)$ with $\lambda_t^e \circ q_R = W(q_R) \circ \lambda_t$. The
homomorphism $\lambda_t^e$ defines a normed pre-$\lambda$-ring structure on $R^e$,
and the identity $\lambda_t^e \circ q_R = W(q_R) \circ \lambda_t$ says exactly
that $q_R$ is a morphism of pre-$\lambda$-rings. Since $R^e$ is an $E$-ring,
Proposition~\ref{prop:e-ring-pre-lambda-is-lambda} shows that this structure is a
$\lambda$-ring structure.
\end{proof}

\begin{corollary}
\label{cor:R^e-binomial}
If $R$ admits a normed pre-$\lambda$-ring structure, then $R^e$ is binomial.
\end{corollary}
\begin{proof}
This follows from Propositions~\ref{prop:lambda-on-R^e}
and~\ref{prop:e-ring-pre-lambda-is-lambda}.
\end{proof}

\begin{proposition}
\label{prop:seven-saturation}
	Fix $\sqrt{2}\in \bb{Z}_7$, and let $\mathcal{E}$ be the $7$-saturation of
	$\bb{Z}[\sqrt{2}]$ in $\bb{Z}_7$:
\begin{equation}
\mathcal{E} = \{x \in \bb{Z}_7 \mid \exists\, n\geq 0 : 7^n x \in \bb{Z}[\sqrt{2}]\}.
\end{equation}
Then $\mathcal{E}$ is a countable, $\bb{Z}$-torsion-free $E$-ring that admits no
normed pre-$\lambda$-ring structure. Moreover, $\mathcal{E}$ is not solid.
\end{proposition}
\begin{proof}
	An element $\sqrt{2}\in \bb{Z}_7$ exists by Hensel's lemma, since
	$3^2\equiv 2 \pmod 7$. As $\mathcal{E}$ is contained in the image of
	$\bb{Z}[1/7,\sqrt{2}]$ in $\bb{Q}_7$, it is countable, and it is
	$\bb{Z}$-torsion-free because it is a subgroup of $\bb{Z}_7$.

	We first show that $\mathcal{E}$ is an $E$-ring. The subgroup
	$\mathcal{E}\subseteq\bb{Z}_7$ is $7$-saturated, i.e.\ $\mathcal{E}\cap
	7^m\bb{Z}_7 = 7^m\mathcal{E}$: indeed, if $x\in \mathcal{E}\cap 7^m\bb{Z}_7$
	then $7^{-m}x\in\bb{Z}_7$, and if $7^n x\in\bb{Z}[\sqrt{2}]$ then
	$7^{n}(7^{-m}x)\cdot 7^m = 7^n x\in\bb{Z}[\sqrt{2}]$, so $7^{-m}x\in\mathcal{E}$.
	Consequently the subspace topology on $\mathcal{E}$ from $\bb{Z}_7$ coincides
	with its $7$-adic topology, so every endomorphism $f:\mathcal{E}^+\to
	\mathcal{E}^+$ is continuous for this topology. Since $\mathcal{E}$ is dense
	in $\bb{Z}_7$, the map $f$ extends to a continuous endomorphism
	$\hat{f}:\bb{Z}_7\to \bb{Z}_7$. As $\bb{Z}_7$ is an $E$-ring,
	$\hat f(x)=f(1)\, x$ for all $x\in \bb{Z}_7$; in particular $f(x)=f(1)\, x$
	for all $x\in \mathcal{E}$. Thus $\mathcal{E}$ is an $E$-ring.

	The element $\sqrt{2}$ lies in $\mathcal{E}$, but
	$\binom{\sqrt{2}}{2}=\frac{\sqrt 2(\sqrt 2-1)}{2}=\frac{2-\sqrt{2}}{2}$ does
	not. Indeed, otherwise $7^n(2-\sqrt{2})\in 2\,\bb{Z}[\sqrt{2}]$ for some
	$n\geq 0$; reducing modulo $2$, this forces the class of $\sqrt{2}$ to vanish
	in $\bb{Z}[\sqrt 2]/2\cong \bb{F}_2[u]/(u^2)$, which is false. Hence
	$\mathcal{E}$ is not binomial, and so by
	Proposition~\ref{prop:e-ring-pre-lambda-is-lambda} it admits no normed
	pre-$\lambda$-ring structure.

    Finally, we show that $\mathcal{E}$ is not solid. Suppose it were. Then the
    map $\mathcal{E} \to \mathcal{E}\otimes_{\bb{Z}} \mathcal{E}$, $x\mapsto
    x\otimes 1$, would be an isomorphism. Since $\mathcal{E}\otimes_{\bb{Z}}
    \bb{Q}\cong \bb{Q}[\sqrt{2}]$, tensoring with $\bb{Q}$ would make
    $\bb{Q}[\sqrt{2}] \to \bb{Q}[\sqrt{2}] \otimes_{\bb{Q}} \bb{Q}[\sqrt{2}]$,
    $x\mapsto x\otimes 1$, an isomorphism. This is impossible, since
    $\bb{Q}[\sqrt{2}]$ is a $\bb{Q}$-vector space of dimension $2$, so the
    codomain has dimension $4$. Therefore $\mathcal{E}$ is not solid.
\end{proof}

\begin{corollary}
\label{cor:E-localisation}
The $\mathcal{E}$-localization is not weakly right exact.
\end{corollary}
\begin{proof}
Since $\mathcal{E}$ is an $E$-ring that is not solid
(Proposition~\ref{prop:seven-saturation}), this follows from
Proposition~\ref{prop:weakly right and solid}.
\end{proof}

\section{\texorpdfstring{$R$}{R}-localization of free groups}
In this section we compute the $R$-localization of a free group, for a torsion-free $E$-ring $R$. For a set $X$,
we write $F(X)$ for the free group on $X$ and $F_R(X)$ for
the free $R$-group on $X$. 

\begin{proposition}
\label{prop:free-e-ring-groups-local}
Let $R$ be an $E$-ring such that $R^+$ is torsion-free and let $X$ be a set. Then $F_{R}(X)$
is $R$-local, and the map $F(X) \to F_{R}(X)$ induces 
isomorphisms
\begin{equation}
L_R(F(X)) \cong F_{R}(X).
\end{equation}
\end{proposition}
\begin{proof}
If $X=\varnothing$ there is nothing to prove, and if $|X|=1$ then
$F_{R}(X)\cong R^+ \cong L_R(\bb{Z})$. Assume henceforth that
$|X|\geq 2$. By~\cite[Prop.~16]{myasnikov1996exponential}, the
centralizer of every nontrivial element of $F_R(X)$ is isomorphic
to $R^+$.

We first show that $F_{R}(X)$ is $R$-local. As it is an $R$-group,
it suffices, by Theorem~\ref{thm:r-local-r-groups}, to show that it is
$R/1$-null. Suppose not, and let $f: R/1\to F_{R}(X)$ be a
nontrivial homomorphism. Choose a nontrivial element $x$ in the image of $f$. Since the image is abelian it is contained in the centralizer $\mathsf{Cent}(x)\cong R^+$. Thus
$f$ corestricts to a nontrivial homomorphism $R/1 \to R^+$, which is
impossible, since $R^+$ is $R/1$-null. Therefore $F_{R}(X)$ is
$R$-local.

Now let $G$ be any $R$-local group and $\varphi:F(X)\to G$ a
homomorphism. Since $G$ is $R$-local, it admits an $R$-group structure
(Proposition~\ref{prop:R-ring_structure_on_a_local_group}), so $\varphi$ lifts uniquely to an
$R$-homomorphism $\tilde \varphi:F_{R}(X)\to G$. Moreover, by
Proposition~\ref{prop:R-ring_structure_on_a_local_group}, every homomorphism
$F_{R}(X)\to G$ is automatically an $R$-homomorphism, so
$\tilde\varphi$ is the unique homomorphism extending $\varphi$. Hence
$F_{R}(X)$ satisfies the universal property of the $R$-localization
of $F(X)$.
\end{proof}

\begin{corollary}
\label{cor:HP:free}
Let $R$ be a binomial $E$-ring and let $X$ be a set. Then $F_R(X)$ is a Hall-Petresco $R$-group and the map $F(X)\to F_R(X)$ induces an isomorphism 
\begin{equation}
C_R^{\mathrm{HP}}(F(X)) \cong F_R(X).
\end{equation}
\end{corollary}
\begin{proof}
Since $F_R(X)$ is $R$-local, it is Hall-Petresco (Proposition \ref{prop:local are Hall--Petresco}). Then the universal property of $F_R(X)$ among $R$-groups implies the universal property among Hall-Petresco $R$-groups. 
\end{proof}

\begin{corollary}
	\label{cor:free-zp-localization-pro-p}
	The $\bb{Z}_p$-localization of $F(X)$ embeds into the pro-$p$
	completion of $F(X)$.
\end{corollary}
\begin{proof}
	Since $\bb{Z}_p$ is an $E$-ring, this follows from Proposition~\ref{prop:free-e-ring-groups-local}
	and~\cite[Cor.~1.3]{jaikin2024free}.
\end{proof}

\begin{definition}
For a ring $R$ and a set $X$, we write $R\langle\!\langle X\rangle\!\rangle$ for
the algebra of formal power series in the non-commuting variables $X$, and we
denote by
\begin{equation}
\Delta_R(X) \subseteq R\langle\!\langle X\rangle\!\rangle
\end{equation}
the kernel of the augmentation $R\langle\!\langle X\rangle\!\rangle \to R$ sending
each variable in $X$ to zero. Then $1+\Delta_R(X)$ is a group under
multiplication.
\end{definition}

\begin{corollary}
	\label{cor:free-binomial-e-ring-magnus}
	Let $R$ be a binomial domain that is an $E$-ring, and let $X$ be a set. Then
	$L_R(F(X))$ embeds into the group $1+\Delta_R(X)$.
\end{corollary}
\begin{proof}
This follows from Proposition~\ref{prop:free-e-ring-groups-local}
and~\cite[Th.~5.8]{jaikin2026correction}.
\end{proof}

\section{\texorpdfstring{$R$}{R}-localization of torsion-free nilpotent groups}

In this section we describe the $R$-localization of a finitely generated
torsion-free nilpotent group in terms of Mal'cev bases.

\begin{definition}[Poly-infinite cyclic central series]
Let $G$ be a finitely generated torsion-free nilpotent group. A \emph{poly-infinite
cyclic central series} in $G$ is a central series
\begin{equation}
G=G_1 \supset G_2 \supset \dots \supset G_{k+1} = 1
\end{equation}
with $G_i/G_{i+1}\cong \bb{Z}$ for all $i$. Every finitely generated torsion-free
nilpotent group has such a series
(see~\cite[Th.~4.5]{clement_theory_2017}).
\end{definition}

\begin{definition}[Mal'cev basis]
Let $G$ be a finitely generated torsion-free nilpotent group. A \emph{Mal'cev
basis} of $G$ is a generating set $u_1,\dots,u_k\in G$ for which the subgroups
$G_i=\langle u_i,\dots,u_k \rangle$ form a poly-infinite cyclic central series.
Every poly-infinite cyclic central series arises from a Mal'cev basis in this
way. Each element $g\in G$ then has a unique expression $g=u_1^{a_1}\cdots
u_k^{a_k}$ with $a_1,\dots,a_k\in \bb{Z}$, and the isomorphism
$G_i/G_{i+1}\cong \bb{Z}$ is given by $u_i^{a_i}\cdots u_k^{a_k}\mapsto a_i$.
Moreover, there are integer-valued polynomials $p_1,\dots,p_k\in
\bb{Q}[x_1,\dots,x_k,y_1,\dots,y_k]$ describing the product,
\begin{equation}
(u_1^{a_1}\cdots u_k^{a_k}) \cdot (u_1^{A_{1,3}} \cdots u_k^{\beta_k})
=
u_1^{p_1(\bar a,\bar b)} \cdots u_k^{p_k(\bar a,\bar b)},
\end{equation}
and integer-valued polynomials $q_1,\dots,q_k\in \bb{Q}[x_1,\dots,x_k,z]$
describing the powers,
\begin{equation}
(u_1^{a_1} \cdots u_k^{a_k})^m = u_1^{q_1(\bar a, m)} \cdots u_k^{q_k(\bar a,m)}
\qquad (m\in \bb{Z})
\end{equation}
(see~\cite[Th.~4.9]{clement_theory_2017}).
\end{definition}

\begin{definition}[Hall $R$-completion]
\label{def:Hall completion}
Let $R$ be a binomial ring and let $G$ be a finitely generated torsion-free
nilpotent group with Mal'cev basis $u_1,\dots,u_k$. Since the polynomials $p_i$
and $q_i$ are integer-valued, each can be written as an integral linear
combination of products of binomial polynomials, and hence defines polynomial
maps $\tilde p_i:R^{2k} \to R$ and $\tilde q_i:R^{k+1}\to R$. The \emph{Hall
$R$-completion} of $G$ is the group $H_R(G)$ whose elements are the formal products
$u_1^{r_1}\cdots u_k^{r_k}$ with $r_i\in R$, with multiplication
\begin{equation}
(u_1^{r_1}\cdots u_k^{r_k}) \cdot (u_1^{s_1} \cdots u_k^{s_k})
=
u_1^{\tilde p_1(\bar r,\bar s)} \cdots u_k^{\tilde p_k(\bar r,\bar s)}
\end{equation}
(see~\cite[Th.~4.11]{clement_theory_2017}). Then $H_R(G)$ is an $R$-group, with
$R$-power operation
\begin{equation}
(u_1^{r_1} \cdots u_k^{r_k})^s = u_1^{\tilde q_1(\bar r, s)} \cdots u_k^{\tilde q_k(\bar r,s)}.
\end{equation}
Moreover, $H_R(G)$ has a central series given by the subgroups
\begin{equation}
\tilde G_i=\{u_i^{r_i} \cdots u_k^{r_k} \mid r_i,\dots,r_k\in R \}, \qquad 1\leq i\leq k+1,
\end{equation}
for which $\tilde G_i/\tilde G_{i+1} \to R^+$, $u_i^{r_i} \cdots u_k^{r_k}\mapsto
r_i$, is an isomorphism of $R$-groups.
\end{definition}

\begin{theorem}[{\cite[Th.~4.23]{clement_theory_2017}}]
\label{th:Hall}
Let $R$ be a binomial ring and let $G$ be a finitely generated torsion-free
nilpotent group. Then $H_R(G)$ is $R$-powered nilpotent in the sense of Hall, and
the homomorphism
\begin{equation}
G \to H_R(G), \qquad u_i\mapsto u_i,
\end{equation}
is the universal homomorphism from $G$ to an $R$-powered nilpotent group in the
sense of Hall.
\end{theorem}

\begin{example}
If $R = \bb{Z}_p$ is the ring of $p$-adic integers, then $H_{\bb{Z}_p}(G)$
coincides with the $p$-adic completion of
$G$~\cite[p.~28]{clement_theory_2017}.
\end{example}

\begin{lemma}
	\label{lemma:linfty-hall-completion}
	Let $R$ be a binomial ring and let $G$ be a finitely generated torsion-free
	nilpotent group. Then $H_R(G)$ is a Hall--Petresco $R$-group, and the map
	$G\to H_R(G)$ induces an isomorphism
    \begin{equation}
    C^{\mathrm{HP}}_R(G) \cong H_R(G).
    \end{equation}
\end{lemma}
\begin{proof}
Since $R$ is binomial, a nilpotent Hall--Petresco $R$-group is the same as an
$R$-powered nilpotent group in the sense of Hall
(Remark~\ref{rem:R-powered}). By
Proposition~\ref{prop:F^{(n-1)}(G) is nilpotent}, $C^{\mathrm{HP}}_R(G)$ is
nilpotent, hence such a group; the isomorphism then follows by comparing the
universal properties of $C^{\mathrm{HP}}_R(G)$ and $H_R(G)$
(Theorem~\ref{th:Hall}).
\end{proof}

\begin{proposition}
	\label{prop:e-ring-localization-hall-completion}
	Let $R$ be a commutative ring admitting a normed pre-$\lambda$-ring structure,
	and let $G$ be a finitely generated torsion-free nilpotent group. Then $R^e$
	is binomial, $H_{R^e}(G)$ is $R$-local, and the map $G\to H_{R^e}(G)$ induces
	an isomorphism
    \begin{equation}
    L_R(G) \cong H_{R^e}(G).
    \end{equation}
\end{proposition}
\begin{proof}
By Corollary~\ref{cor:R^e-binomial}, $R^e$ is binomial. The group $H_{R^e}(G)$ has
a central series whose successive quotients are isomorphic to $(R^e)^+$. Since
$(R^e)^+$ is $R/1$-null and the class of $R/1$-null groups is closed under
extensions (Lemma~\ref{lemma:null-closed-under-extensions}), $H_{R^e}(G)$ is
$R/1$-null.

By Lemma~\ref{lemma:linfty-hall-completion}, $C_{R^e}^{\mathrm{HP}}(G) \cong
H_{R^e}(G)$, so $C_{R^e}^{\mathrm{HP}}(G)$ is $R/1$-null. Now $H_{R^e}(G)$ is an
$R^e$-group that is $R/1$-null; since the homomorphism $R^e/1 \to R/1$ induced by
$q_R$ is surjective (as $q_R$ is, by
Proposition~\ref{prop:R->R^e-surjective}), $H_{R^e}(G)$ is also $R^e/1$-null.
Hence $H_{R^e}(G)$ is $R^e$-local by Proposition~\ref{prop:r-null}, and
Corollary~\ref{cor:L_R is a quotient of HP} gives $L_{R^e}(G)\cong H_{R^e}(G)$.
Since this group is $R/1$-null, Proposition~\ref{prop:epi-L_E->L_R} yields
$L_R(G)\cong N_{R/1}(L_{R^e}(G)) \cong H_{R^e}(G)$.
\end{proof}

\begin{corollary}
Let $R$ be a binomial $E$-ring and let $G$ be a finitely generated torsion-free
nilpotent group. Then $H_{R}(G)$ is $R$-local, and the map $G\to H_{R}(G)$ induces
an isomorphism
    \begin{equation}
    L_R(G) \cong H_{R}(G).
    \end{equation}
\end{corollary}
\begin{proof}
A binomial $E$-ring admits a normed pre-$\lambda$-ring structure
(Proposition~\ref{prop:e-ring-pre-lambda-is-lambda}) and satisfies $R\cong R^e$,
since it is an $E$-ring. The claim is therefore the special case $R=R^e$ of
Proposition~\ref{prop:e-ring-localization-hall-completion}.
\end{proof}

\section{Twists of \texorpdfstring{$R$}{}-group structures} 

In this section, we introduce a method for modifying $R$-group structures.
The method uses derivations and twisting maps to twist the original
exponentiation operations, producing new $R$-group structures on the same
underlying group. As an application, we construct a nilpotent
$\bb Z_p$-group of class $3$ whose center is not a $\bb Z_p$-ideal.

\begin{definition}[Derivation]
Let $M$ be an $R$-module. A function $\delta:R\to M$ is called a derivation
if it is a homomorphism of additive groups and satisfies the Leibniz rule
$\delta(rr')=r\delta(r')+\delta(r)r'$
for all $r,r'\in R$. 
\end{definition}

\begin{lemma}
\label{lemma:non-trivial_derivation} 
There is a nontrivial derivation $\bb Q_p\to \bb Q_p$ and a nontrivial
derivation $\bb Z_p\to \bb Q_p$.
\end{lemma}

\begin{proof}
Let $\Omega_{\bb Q_p/\bb Q}$ be the module of Kähler differentials. Since
$\operatorname{char}\bb Q=0$, every transcendence basis of $\bb Q_p/\bb Q$
is separating. Since $\bb Q_p$ is uncountable, 
$\mathrm{trdeg}_{\bb Q}\bb Q_p\neq 0$. Therefore
$\dim_{\bb Q_p}\Omega_{\bb Q_p/\bb Q}\neq 0$
by \cite[Th.~26.5]{MatsumuraCRT}. Hence
\begin{equation}
\mathrm{Der}_{\bb Q}(\bb Q_p,\bb Q_p)
\cong 
\mathrm{Hom}_{\bb Q_p}(\Omega_{\bb Q_p/\bb Q},\bb Q_p)
\end{equation}
is nontrivial. It follows that there is a nontrivial derivation
$\bb Q_p\to \bb Q_p$. It is easy to see that its restriction to $\bb Z_p$
is a nontrivial derivation $\bb Z_p\to \bb Q_p$.
\end{proof}

\begin{definition}[Twisting map]
Let $G$ be an $R$-group. A function $\eta:G\to G$ is called an
$R$-twisting map if the following conditions hold for all $x,y\in G$ and
$r\in R$:
\begin{enumerate}
    \item[(C1)] $\tau(x^r)=\tau(x)^r$;
    \item[(C2)] $\tau(x^y)=\tau(x)^y$;
    \item[(C3)] if $[x,y]=1$, then $\tau(xy)=\tau(x)\tau(y)$;
    \item[(C4)] $\tau(\tau(x))=1$.
\end{enumerate}
\end{definition}

\begin{proposition}[$R$-group structure twist]
Let $R$ be a commutative ring, let $G$ be an $R$-group, let
$\delta:R\to R$ be a derivation, and let $\tau:G\to G$ be an
$R$-twisting map. Then the operation $G\times R\to G$,
$(x,r)\mapsto x^{(r)}$, defined by
\begin{equation}
x^{(r)}=x^r\tau(x)^{\delta(r)},    
\end{equation}
where $(-)^r$ denotes the original $R$-group operation, defines an 
$R$-group structure on $G$.
\end{proposition}

\begin{proof}
This follows from the more general Proposition~\ref{prop:rel-twist}. 
\end{proof}

\begin{example}
Let $\bb K$ be a field of characteristic $0$, let $N$ be the free nilpotent
group of class $2$ on two generators $x,y$, and let $N^{\bb K}$ be its Hall
$\bb K$-completion. Its elements have the form
$x^a y^b [x,y]^c,$
where $a,b,c\in \bb K.$
Assume that
$
f:\bb K^2\to \bb K
$
is a map which is $\bb K$-linear on every $1$-dimensional vector subspace
of $\bb K^2$. We claim that the map
$\tau:N^{\bb K}\to N^{\bb K}$ defined by
\begin{equation}
\tau(x^ay^b[x,y]^c)=[x,y]^{f(a,b)}
\end{equation}
is a $\bb K$-twisting map. The only non-obvious axiom to check is (C3).
It follows from the fact that $x^ay^b[x,y]^c$ and
$x^{a'}y^{b'}[x,y]^{c'}$ commute if and only if $(a,b)$ and $(a',b')$
are linearly dependent in $\bb K^2$. As a corollary, for any derivation
$\delta:\bb K\to \bb K$, the formula
\begin{equation}
(x^ay^b [x,y]^c)^{(r)}
=
x^{ar} y^{br} [x,y]^{cr-\binom{r}{2}ab+f(a,b)\delta(r)}
\end{equation}
defines a $\bb K$-group structure on $N^{\bb K}$. 
\end{example}

\begin{definition}[Relative twisting map] 
Let $R\subseteq S$ be a commutative ring extension, let $H$ be an
$S$-group, and let $G\subseteq H$ be its $R$-subgroup. A function
$\tau:G\to G$ is called a relative $(R,S)$-twisting map if it is an
$R$-twisting map and the following properties are satisfied for all
$x\in G$ and $s\in S$:
\begin{enumerate}
    \item[(C5)] $\tau(x)^s\in G$;
    \item[(C6)] $\tau(\tau(x)^s)=1$. 
\end{enumerate}
\end{definition}

\begin{remark}
If $R=S$ and $G=H$, then a relative $(R,S)$-twisting map is just an
$R$-twisting map. 
\end{remark}

\begin{proposition}[Relative $R$-group structure twist]
\label{prop:rel-twist}
Let $R\subseteq S$ be an extension of commutative rings, let $H$ be an
$S$-group, and let $G\subseteq H$ be an $R$-subgroup. Let
$\delta:R\to S$ be a derivation, and let $\tau:G\to G$ be a relative
$(R,S)$-twisting map. Then the operation $G\times R\to G$,
$(x,r)\mapsto x^{(r)}$, defined by
\begin{equation}
x^{(r)}=x^r\tau(x)^{\delta(r)},    
\end{equation}
where $(-)^r$ denotes the original $R$-group operation, defines an 
$R$-group structure on $G$.
\end{proposition}

\begin{proof}
First observe that condition (C2), applied with $y=x$, gives
$
\tau(x)=\tau(x^x)=\tau(x)^x.$
Hence
$[\tau(x),x]=1.$
Therefore the elements $x^s$ and $\tau(x)^{s'}$ commute for all
$s,s'\in S$.

\emph{(R1).} Let $r,r'\in R$ and $x\in G$. Since $\delta$ is additive, we
obtain
\begin{equation}
x^{(r+r')}
=x^{r+r'}\tau(x)^{\delta(r+r')}
=x^r\tau(x)^{\delta(r)}
x^{r'}\tau(x)^{\delta(r')}
=x^{(r)}x^{(r')}.
\end{equation}

\emph{(R2).} Since $\delta$ is a derivation, $\delta(1)=0$. Therefore
$x^{(1)}=x^1\tau(x)^{\delta(1)}=x.$

\emph{(R3).} Since $x^r$ commutes with $\tau(x)^{\delta(r)}$, (R5) gives
\begin{equation}
\bigl(x^r\tau(x)^{\delta(r)}\bigr)^{r'}
=
x^{rr'}\tau(x)^{\delta(r)r'}.
\end{equation}
Since $x^r$ and $\tau(x)^{\delta(r)}$ commute, and both belong to $G$,
condition (C3) gives
\begin{equation}
\tau\bigl(x^r\tau(x)^{\delta(r)}\bigr)
=
\tau(x^r)\tau\bigl(\tau(x)^{\delta(r)}\bigr).
\end{equation}
By (C1),
$\tau(x^r)=\tau(x)^r,$
and by (C6),
$\tau\bigl(\tau(x)^{\delta(r)}\bigr)=1.$
Therefore
\begin{equation}
\tau\bigl(x^r\tau(x)^{\delta(r)}\bigr)=\tau(x)^r.
\end{equation}
Hence
\begin{equation}
\begin{aligned}
\bigl(x^{(r)}\bigr)^{(r')}
&=
\bigl(x^r\tau(x)^{\delta(r)}\bigr)^{r'}
\tau\bigl(x^r\tau(x)^{\delta(r)}\bigr)^{\delta(r')} \\
&=
x^{rr'}\tau(x)^{\delta(r)r'+r\delta(r')}
=
x^{rr'}\tau(x)^{\delta(rr')}
=
x^{(rr')}.
\end{aligned}
\end{equation}

\emph{(R4).} We prove the conjugation axiom. Let $x,y\in G$ and $r\in R$.
Using the original conjugation axiom in $H$ and condition (C2), we get
\begin{equation}
(x^y)^{(r)}
=
(x^y)^r\tau(x^y)^{\delta(r)}
=
(x^r)^y\bigl(\tau(x)^{\delta(r)}\bigr)^y
=
\bigl(x^r\tau(x)^{\delta(r)}\bigr)^y
=
(x^{(r)})^y.
\end{equation}

\emph{(R5).} Finally, suppose that $[x,y]=1$. Since $x$ and $y$ commute,
condition (C2) gives
$\tau(x)=\tau(x^y)=\tau(x)^y,$
so $\tau(x)$ commutes with $y$. Similarly, $\tau(y)$ commutes with $x$.
Moreover, since $xy=yx$, we have 
$
\tau(x)\tau(y)=\tau(xy)=\tau(yx)=\tau(y)\tau(x),
$
so $\tau(x)$ and $\tau(y)$ commute. Therefore all elements occurring below
commute in the required way, and (R5) in the $S$-group $H$ gives
\begin{equation}
\begin{aligned}
(xy)^{(r)}
&=
(xy)^r\tau(xy)^{\delta(r)}
=
x^r y^r\bigl(\tau(x)\tau(y)\bigr)^{\delta(r)} \\
&=
x^r y^r\tau(x)^{\delta(r)}\tau(y)^{\delta(r)}
=
x^r\tau(x)^{\delta(r)}
y^r\tau(y)^{\delta(r)}
=
x^{(r)}y^{(r)}.
\end{aligned}
\end{equation}
Thus the operation $(x,r)\mapsto x^{(r)}$ satisfies the axioms
(R1)--(R5), and hence defines an $R$-group structure on $G$.
\end{proof}

\begin{notation}[The group $\mathcal{G}$]
Denote by $N$ the free nilpotent group of class $2$ generated by $x,y$.
Every element of $N$ is uniquely written as
\begin{equation}
x^ay^b[x,y]^c,
\qquad a,b,c\in \bb Z.
\end{equation}
Consider the direct product $\langle w\rangle\times N$ with an infinite
cyclic group $\langle w\rangle$, and define an endomorphism
$\alpha:\langle w\rangle\times N\to \langle w\rangle\times N$
by
\begin{equation}
w\mapsto w[x,y],
\hspace{5mm}
x\mapsto xy^p,
\hspace{5mm}
y\mapsto y.
\end{equation}
Note that $\alpha([x,y])=[x,y]$. One checks that this endomorphism is an
automorphism and that $\alpha^n$ is given by
\begin{equation}
w\mapsto w[x,y]^n,
\hspace{5mm}
x\mapsto xy^{pn},
\hspace{5mm}
y\mapsto y
\end{equation}
for every $n\in\bb Z$. Therefore $\alpha$ defines an action of the cyclic
group $\langle v\rangle$ on $\langle w\rangle\times N$. We define
$\mathcal{G}$ as the semidirect product
\begin{equation}
\mathcal{G}=\langle v\rangle\ltimes(\langle w\rangle\times N)
\end{equation}
such that
$v^{-1}hv=\alpha(h)$
for every $h\in\langle w\rangle\times N$. Every element of $\mathcal{G}$
can be written uniquely as
\begin{equation}
v^aw^bx^cy^d[x,y]^e,
\qquad a,b,c,d,e\in \bb Z.
\end{equation}
\end{notation}

\begin{lemma}
The group $\mathcal{G}$ is a torsion-free nilpotent group of class $3$, and
$v,w,x,y,[x,y]$ is a Malcev basis. The product in $\mathcal{G}$ is given by
\begin{equation}
\label{eq:G-product}
\begin{aligned}
&\left(v^a w^b x^c y^d [x,y]^e\right)
\left(v^{a'} w^{b'} x^{c'} y^{d'} [x,y]^{e'}\right) \\
&\qquad =
v^{a+a'}
w^{b+b'}
x^{c+c'}
y^{d+d'+pa'c}
[x,y]^{
e+e'
+a'b
-pa'\binom{c}{2}
-c'(d+pa'c)
}.
\end{aligned}
\end{equation}
The $n$-th power in $\mathcal{G}$ is given by
\begin{equation}
\label{eq:G-power}
\begin{aligned}
\left(v^a w^b x^c y^d [x,y]^e\right)^n
&=
v^{na}
w^{nb}
x^{nc}
y^{nd+pac\binom{n}{2}}  \\
&\quad\cdot
[x,y]^{
ne
+(ab-cd)\binom{n}{2}
-pa\left(\binom{c}{2}+c^2\right)\binom{n}{2}
-2pac^2\binom{n}{3}
}.
\end{aligned}
\end{equation}
\end{lemma}

\begin{proof}
The formulas are obtained by an explicit computation. Using these formulas,
it is easy to see that $v,w,x,y,[x,y]$ is a Malcev basis. Moreover, it is
easy to see that
$\mathcal{G} \supset  \langle y,[x,y] \rangle \supset
\langle [x,y] \rangle \supset 1$
is a central series. Therefore $\mathcal G$ is nilpotent of class $3$.
\end{proof}

\begin{notation}[The group $\mathcal{G}'$]
Let $H_{\bb Q_p}(\mathcal G)$ denote the Hall $\bb Q_p$-completion of
$\mathcal G$ (Definition~\ref{def:Hall completion}). Every element of
$H_{\bb Q_p}(\mathcal G)$ has a unique expression of the form
\begin{equation}
v^a w^b x^c y^d [x,y]^e,
\qquad
a,b,c,d,e\in \bb Q_p,
\end{equation}
where the product is given by formula~\eqref{eq:G-product}, and the
$\bb Q_p$-group structure is given by formula~\eqref{eq:G-power}. We define
\begin{equation}
\mathcal{G}'\subseteq H_{\bb Q_p}(\mathcal{G})
\end{equation}
to be the subgroup consisting of all elements of the form
\begin{equation}
v^a w^b x^c y^d [x,y]^e,
\qquad
a,c,d\in \bb Z_p,\quad b,e\in \bb Q_p.
\end{equation}
The product and power formulas show that $\mathcal{G}'$ is a
$\bb Z_p$-subgroup of $H_{\bb Q_p}(\mathcal G)$ with respect to the
restricted $\bb Z_p$-powers. Since $\mathcal G$ is nilpotent of class $3$, its Hall $\bb Q_p$-completion
$H_{\bb Q_p}(\mathcal G)$ is also nilpotent of class $3$
(Lemma~\ref{lemma:linfty-hall-completion},
Proposition~\ref{prop:F^{(n-1)}(G) is nilpotent}). Hence $\mathcal G'$ is
nilpotent of class $3$. 
\end{notation}

\begin{notation}[The relative twisting map $\tau:\mathcal{G}'\to \mathcal{G'}$]
Consider two maps
$f,g:\bb Z_p^2\to \bb Z_p$
defined as follows:
\begin{equation}
f(c,d)=
\begin{cases}
c, & c\neq 0,\ d/c\in 1+p\bb Z_p,\\
0, & \text{otherwise},
\end{cases}
\hspace{3mm}
g(c,d)=
\begin{cases}
(d-c)/p, & c\neq 0,\ d/c\in 1+p\bb Z_p,\\
0, & \text{otherwise}.
\end{cases}
\end{equation}
Define a map
\begin{equation}
\tau:\mathcal{G}'\to \mathcal{G}'
\end{equation}
by the formula
\begin{equation}
\label{eq:twist_example}
\tau(v^a w^b x^c y^d [x,y]^e)
=
\begin{cases}
w^{f(c,d)} [x,y]^{g(c,d)}, & a=0,\\
1, & a\neq 0.
\end{cases}
\end{equation}
\end{notation}

\begin{lemma}
\label{lemma:twisting}
The map $\tau:\mathcal{G}'\to \mathcal{G}'$ is a relative
$(\bb Z_p,\bb Q_p)$-twisting map.
\end{lemma}

\begin{proof}
First note that the functions $f$ and $g$ satisfy the following two
properties. For every $r\in\bb Z_p$,
\begin{equation}
f(rc,rd)=rf(c,d),
\qquad
g(rc,rd)=rg(c,d).
\end{equation}
Moreover, they are additive on every one-dimensional $\bb Q_p$-subspace
of $\bb Q_p^2$ after intersection with $\bb Z_p^2$. Indeed, on a line
$c=0$ both functions are zero. On a line with fixed slope $d/c$, either the
slope is not in $1+p\bb Z_p$, in which case both functions are zero, or
the slope is $1+p\lambda$ with $\lambda\in\bb Z_p$, in which case
\begin{equation}
f(c,d)=c,
\qquad
g(c,d)=\frac{d-c}{p}=\lambda c,
\end{equation}
which are additive in $c$. We shall also use the shift identities
\begin{equation}
f(c,d+prc)=f(c,d),
\qquad
g(c,d+prc)=g(c,d)+rf(c,d)
\end{equation}
for every $r\in\bb Z_p$.

\emph{Condition (C1).}
Let
$A=v^a w^b x^c y^d[x,y]^e\in \mathcal G',$
$r\in\bb Z_p.$
If $a\neq0$, then $\tau(A^r)=\tau(A)^r=1$. If $a=0$, then the homogeneity
of $f$ and $g$ gives
\begin{equation}
\tau(A^r)
=
w^{f(rc,rd)}[x,y]^{g(rc,rd)}
=
w^{rf(c,d)}[x,y]^{rg(c,d)}
=
\tau(A)^r.
\end{equation}
Thus (C1) holds.

\emph{Condition (C2).}
Conjugation preserves the $v$-coordinate. Hence if $a\neq0$, then
$\tau(A^B)=\tau(A)^B=1$. Assume now that $a=0$. Conjugation by elements of
the base subgroup $\langle w,x,y,[x,y]\rangle$ changes only the central
$[x,y]$-coordinate of $A$, and therefore does not change the pair $(c,d)$.
Since
$\tau(A)=w^{f(c,d)}[x,y]^{g(c,d)}$
lies in the abelian subgroup $\langle w,[x,y]\rangle_{\bb Q_p}$,
conjugation by the base subgroup does not affect the required equality.

It remains to check conjugation by $v^r$. Using the formula above,
$A^{v^r}$ has $(x,y)$-coordinates
$(c,d+prc).$
Therefore
$\tau(A^{v^r})
=
w^{f(c,d+prc)}[x,y]^{g(c,d+prc)}.$
By the shift identities,
\begin{equation}
f(c,d+prc)=f(c,d),
\qquad
g(c,d+prc)=g(c,d)+rf(c,d).
\end{equation}
Hence
\begin{equation}
\tau(A^{v^r})
=
w^{f(c,d)}[x,y]^{g(c,d)+rf(c,d)}
=
\tau(A)^{v^r}.
\end{equation}
Thus (C2) holds.

\emph{Condition (C3).}
Let
$A=v^a w^b x^c y^d[x,y]^e$
and
$B=v^{a'}w^{b'}x^{c'}y^{d'}[x,y]^{e'}$
be commuting elements of $\mathcal G'$.

If $a=a'=0$, then $A$ and $B$ lie in the base subgroup. Their commutativity
implies
$cd'-c'd=0.$
Thus $(c,d)$ and $(c',d')$ lie on one $\bb Q_p$-line. By the line-additivity
of $f$ and $g$, we obtain
$\tau(AB)=\tau(A)\tau(B).$

Suppose exactly one of $a,a'$ is nonzero. Without loss of generality,
assume $a'\neq0$ and $a=0$. Since $A$ commutes with $B$, the pair $(c,d)$
must be fixed by the nontrivial shear
$(c,d)\mapsto(c,d+pa'c).$
Since $\bb Z_p$ is an integral domain, this gives
$
c=0.
$
Hence
$
f(c,d)=g(c,d)=0.
$
Thus $\tau(A)=1$, while $\tau(B)=1$ and $\tau(AB)=1$ because $B$ and $AB$
have nonzero $v$-coordinate. Hence (C3) holds in this case.

Finally suppose $a\neq0$ and $a'\neq0$. If $a+a'\neq0$, then
$\tau(A)=\tau(B)=\tau(AB)=1.$
If $a+a'=0$, then $AB$ has zero $v$-coordinate. But $AB$ commutes with $A$,
which has nonzero $v$-coordinate. By the previous paragraph, the
$(x,y)$-coordinates of $AB$ have first coordinate equal to $0$. Therefore
$\tau(AB)=1.$
Again (C3) holds.

\emph{Conditions (C4), (C5), and (C6).}
For every $A\in\mathcal G'$, the element $\tau(A)$ lies in
$\langle w,[x,y]\rangle_{\bb Q_p}.$
Thus, for every $s\in\bb Q_p$,
$\tau(A)^s\in \mathcal G',$
because the $w$- and $[x,y]$-coordinates of $\mathcal G'$ are allowed to
lie in $\bb Q_p$. This proves (C5).

Moreover, every element of $\langle w,[x,y]\rangle_{\bb Q_p}$ has
$v$-, $x$-, and $y$-coordinates equal to zero. Hence its pair $(c,d)$ is
$(0,0)$, and so
$f(0,0)=g(0,0)=0.$
Therefore
$\tau(\tau(A)^s)=1$
for every $s\in\bb Q_p$. This proves (C6), and taking $s=1$ gives (C4).
\end{proof}

\begin{theorem}
\label{th:example:non-ideal}
There exists a $\bb Z_p$-group structure on $\mathcal{G}'$ such that the
center $Z(\mathcal{G}')$ is not a $\bb Z_p$-ideal.
\end{theorem}

\begin{proof}
A direct computation shows that
\begin{equation}
Z(\mathcal{G}')
=
\left\{[x,y]^e \mid e\in \bb Q_p\right\}.
\end{equation}
We define the $\bb Z_p$-group structure on $\mathcal{G}'$ using the relative
$(\bb Z_p,\bb Q_p)$-twisting map from Lemma~\ref{lemma:twisting} by the
formula
$A^{(r)}
=
A^r\tau(A)^{\delta(r)},$
where $\delta:\bb Z_p\to \bb Q_p$ is a nontrivial derivation
(Lemma~\ref{lemma:non-trivial_derivation}) and $A^r$ denotes the ordinary
$\bb Z_p$-group structure on $\mathcal{G}'$ induced from the Hall
$\bb Q_p$-completion $H_{\bb Q_p}(\mathcal G)$. Choose $r\in\bb Z_p$ such
that
$\delta(r)\neq 0.$
Since $\tau(x)=\tau(y)=1$ and $\tau(xy)=w$, we have
\begin{equation}
x^{(r)}=x^r,
\qquad
y^{(r)}=y^r,
\qquad
(xy)^{(r)}=x^r y^r [x,y]^{-\binom{r}{2}}w^{\delta(r)}.
\end{equation}
Then
$
[x,y]\in Z(\mathcal{G}'),
$
but
$
y^{(-r)}x^{(-r)}(xy)^{(r)}
=
[x,y]^{-\binom{r}{2}}w^{\delta(r)}$ is not in
$Z(\mathcal{G}').$
Therefore $Z(\mathcal{G}')$ is not a $\bb Z_p$-ideal with respect to the
twisted $\bb Z_p$-group structure.
\end{proof}

\printbibliography
\end{document}